\documentclass[oneside]{amsart}
\usepackage{amsmath}
\usepackage{amsfonts}
\usepackage{amsthm}
\usepackage{hyperref}
\usepackage{amscd}
\usepackage{color}
\usepackage{enumerate}
\usepackage{xypic}
\usepackage{tikz}
\usepackage{tikz-cd}
\usepackage{colortbl}
\usepackage[normalem]{ulem}
\usepackage{multirow}

\makeatletter
\def\l@subsection{\@tocline{2}{0pt}{2.5pc}{5pc}{}}
\makeatother

\definecolor{color1}{RGB}{255,229,204}
\definecolor{color1a}{RGB}{50,50,50}
\definecolor{color2}{RGB}{178,255,102}
\definecolor{color3}{RGB}{255,153,51}
\definecolor{color4}{RGB}{255,0,0}
\definecolor{color5}{RGB}{102,0,0}
\definecolor{color6}{RGB}{51,102,0}
\definecolor{color7}{RGB}{0,153,153}
\definecolor{color8}{RGB}{0,255,255}
\definecolor{color9}{RGB}{127,0,255}
\definecolor{color10}{RGB}{255,102,255}
\definecolor{color11}{RGB}{166,144,210}

\newcommand{\A}{\mathcal{A}}
\newcommand{\ZZ}{\mathbb{Z}}
\newcommand{\KK}{\mathbb{K}}
\newcommand{\RR}{\mathbb{R}}
\newcommand{\PP}{\mathbb{P}}
\newcommand{\CC}{\mathbb{C}}
\newcommand{\QQ}{\mathbb{Q}}

\newcommand{\dd}{\delta}

\newcommand{\cS}{\mathcal{S}}
\newcommand{\cJ}{\mathcal{J}}
\newcommand{\HH}{\mathcal{H}}
\newcommand{\defect}{\mathbf{def}}
\newcommand{\mult}{\mathbf{m}}
\newcommand{\midp}{\mathrm{mid}}
\newcommand{\HOT}{(\star\star\star)}

\DeclareMathOperator{\codim}{codim}

\DeclareMathOperator{\trop}{Trop}
\DeclareMathOperator{\vol}{Vol}
\DeclareMathOperator{\conv}{conv}

\DeclareMathOperator{\rank}{rank}

\DeclareMathOperator{\sing}{Sing}

\newtheorem{thm}{Theorem}[section]
\newtheorem{prop}[thm]{Proposition}
\newtheorem{lemma}[thm]{Lemma}
\newtheorem{cor}[thm]{Corollary}

\newtheorem{conj}[thm]{Conjecture}

\theoremstyle{definition}

\newtheorem{defn}[thm]{Definition}

\newtheorem{rem}[thm]{Remark}

\newtheorem{oldex}[thm]{Example}

\newenvironment{ex}
  {\pushQED{\qed}
  
  \oldex}
  {\popQED\endoldex}

\newcommand{\hide}[1]{}

\numberwithin{equation}{section}

\newcolumntype{g}{>{\columncolor{yellow}}c}
\title{Projective duals to algebraic and tropical hypersurfaces}

\author{Nathan Ilten}
\address{Department of Mathematics, Simon Fraser University,
8888 University Drive, Burnaby BC V5A1S6, Canada}
\email{nilten@sfu.ca}

\author{Yoav Len}
\address{School of Mathematics, Georgia Institute of Technology, Atlanta GA 30332, USA}
\email{yoav.len@math.gatech.edu}
\thanks{Both authors were supported by the Fields Institute during the \emph{Combinatorial Algebraic Geometry} special semester. The question of the relationship between tropicalization and projective dual curves was first mentioned by Kathl\'en Kohn during the introductory workshop of the semester, and was further discussed during the apprenticeship workshop. We thank Bernd Schober and Kristin Shaw for many useful conversations and comments. In particular, Schober introduced us to Kouchnirenko's result relating Milnor and Newton numbers, and Shaw gave us insight into computing tropical multiplicities. We also thank the anonymous referee for helpful suggestions.}

\begin{document}

\begin{abstract}
We study a tropical analogue of the projective dual variety of a hypersurface. When $X$ is a curve in $\PP^2$ or a surface in $\PP^3$, we  provide an explicit description of $\trop(X^*)$ in terms of $\trop(X)$, as long as $\trop(X)$ is smooth and satisfies a mild genericity condition. 
As a consequence, when $X$ is a curve we describe the transformation of Newton polygons under projective duality, and recover classical formulas for the degree of a dual plane curve.
For higher dimensional hypersurfaces $X$, we give a partial description of $\trop(X^*)$.
\end{abstract}
\maketitle

\tableofcontents 

\section{Introduction}
\subsection{Setting}
Let $X=V(f)\subset \PP^n$ be an irreducible hypersurface of degree $d>1$. The \emph{projective dual} $X^*\subset (\PP^n)^*$ is defined as the Zariski closure in $(\PP^n)^*$ of those hyperplanes $H\subset \PP^n$ such that $H$ is tangent to $X$ at a smooth point \cite[\S 1]{tevelev}. This will be also be a  hypersurface, unless $X$ is covered by lines.
The dual variety $X^*$ captures many interesting aspects of the geometry of $X$. For example, when $X$ is a plane curve with only nodes and cusps, Pl\"ucker's formula relates the degree of $X^*$ to the singularities of $X$:
\begin{equation}\label{eqn:pluecker}
\deg X^*=d(d-1)-2\alpha-3\beta,
\end{equation}
where $\alpha$ is the number of nodes and $\beta$ is the number of cusps.
Dual varieties appear in a wide array of applications, ranging from the study of hypergeometric differential equations \cite{gkz} to optimization \cite{optimization}.

In this paper, we will be working over an algebraically close field $\KK$ of characteristic zero with a non-trivial non-Archimedean valuation $\nu$ which takes trivial values on the integers. Via $\nu$, we have a tropicalization map
\begin{align*}
\trop:(\KK^*)^n&\to \RR^n\\
(x_1,\ldots,x_n)&\mapsto(-\nu(x_1),\ldots,-\nu(x_n)).
\end{align*}
Given a projective variety $Y\subset \PP^n$, its \emph{tropicalization} $\trop(Y)$ is the closure of the image of $Y\cap (\KK^*)^n$ under the tropicalization map. This is a polyhedral complex in $\RR^n$ which remembers many important features of $Y$, including its dimension and degree, see for example \cite{tropical}.

With $X\subset \PP^n$ a hypersurface as above, how does $\trop(X^*)$ depend on $\trop(X)$? We will always assume that $X$ is not a cone, that is, the defining polynomial $f$ of $X$ involves all $n+1$ variables. We make this assumption since otherwise $X^*$ does not meet the torus of $(\PP^n)^*$.
Z.~Izhakian showed that as long as the lowest valuation part of the coefficients of $f$ are sufficiently generic, then $\trop(X^*)$ can be determined directly from $\trop(X)$ \cite[\S 2.1]{izhakian}. The argument can be summarized as follows: the coefficients of the polynomial defining $X^*$ are themselves polynomials in the coefficients of $f$; genericity assumptions ensure that no cancellation occurs, and all arithmetic can be done at the ``tropical'' level. While this argument does guarantee that $\trop(X^*)$ can be recovered from $\trop(X)$, it is not so explicit. However, Izhakian then gives an explicit description in the case of quadric plane curves.

In this paper,  we give explicit combinatorial descriptions of $\trop(X^*)$ in terms of $\trop(X)$ in the cases when $\trop(X)$ is smooth, and $X$ is either a curve or a surface. More precisely, given $p\in \trop(X)$, we say that $q\in \RR^n$ is a \emph{tropical tangent} to $\trop(X)$ at $p$ if there is a smooth point $x\in X\cap (\KK^*)^n$ such that $\trop(x)=p$, and the tangent hyperplane to $X$ at $x$ tropicalizes to $q$ (see Definition \ref{def:tangent}).
The copy of $\RR^n$ containing $q$ is tropically dual to the copy of $\RR^n$ containing $\trop(X)$ in the sense that points of the former correspond to tropical hyperplanes in the latter (and vice versa).
We will show that $\trop(X)$ admits an explicit polyhedral subdivision, such that the collection of tropical tangents to points in the interior of every cell $P$  form a polyhedral complex. Taking the union of these polyhedra, we obtain $\trop(X^*)$. We will also give partial descriptions of $\trop(X^*)$ for higher dimensional hypersurfaces.

Although we will not use most of this terminology in the remainder of the paper, we now explain our approach to the above problem from a conceptual point of view. Consider the conormal variety
\[
	W_X=\overline{\{(x,H)\in\PP^n\times(\PP^n)^*\ |\ x\in X\ \textrm{smooth},\ H\ \textrm{tangent to}\ X\ \textrm{at}\ x\}},
\]
which comes equipped with  projections to $X=V(f)$ and $X^*$. Tropicalizing, we have maps
\[
	\begin{tikzcd}
		&\trop(W_X)\arrow[rd]\arrow[ld]\\
		\trop(X)& & \trop(X^*)
	\end{tikzcd}
\]
whose images are dense in the Euclidean topology.
Hence, we may understand $\trop(X^*)$ by describing $\trop(W_X)$ and then projecting.
While equations for $W_X$ may be easily derived from $f$, they do not form a tropical basis. However, when $X$ is a curve or a surface, we are able to produce a tropical basis for $W_X$ by considering appropriate linear combinations of the ``natural'' equations defining $W_X$. We show that these linear combinations form a tropical basis by using the lifting results of B.~Osserman and S.~Payne \cite{lifting}.

In the remainder of this introduction, we give precise statements for our results on plane curves (\S\ref{sec:curvesintro}) and surfaces and beyond (\S\ref{sec:surfacesintro}), and mention related work (\S\ref{sec:related}). In \S\ref{sec:tangent} we phrase the problem of finding tropical tangents for a hypersurface in terms of solving a system of polynomial equations. We manipulate this system in \S\ref{sec:consistency} to obtain necessary conditions on the set of tropical tangents, and show in \S\ref{sec:lifting} that these conditions are sometimes sufficient, in particular, for situations arising for curves and surfaces.
In \S\ref{sec:NOT}, we analyze the valuations of monomials in $f$ that are not minimal. We then describe tropical tangents at points in an edge of a tropical variety in \S\ref{sec:edges}.
In \S\ref{sec:mult}, we outline a strategy for calculating tropical multiplicities for $\trop(X)$, and do this in several examples.
In \S\ref{sec:curves} we prove our results for curves, and finally in \S\ref{sec:surfaces} we prove our results for surfaces.
We conclude in \S\ref{sec:conclusion} with some thoughts on future research directions.

\subsection{Results: plane curves}\label{sec:curvesintro}
Let $C=\trop(X)\subset \RR^2$ be a smooth tropical curve of degree at least $2$. This already comes equipped with the structure of a polyhedral complex, dual to the subdivision of the Newton polytope of $f$ induced by the valuations of the coefficients of $f$. Zero-dimensional cells in this complex will be called \emph{vertices}, and one-dimensional cells will be called \emph{edges}. Let $e_1,e_2$ be the standard basis for $\RR^2$, and $e_0=-e_1-e_2$. We are using the coordinates $x_1/x_0,x_2/x_0$ on the torus $(\KK^*)^2\subset \PP^2$.

First we give a complete description of $\trop(X^*)$:
\begin{thm}[See \S \ref{sec:curves}]\label{thm:dualcurve}
The tropicalization of the dual curve $\trop(X^*)\subset \RR^2$, with multiplicities, is the union of the following:
\begin{enumerate}
	\item $-E$, where $E$ is an edge of $C$ not parallel to $e_0$, $e_1$, or $e_2$;
	\item $-p-\RR_{\geq 0} \cdot e_i$, where $p$ is a vertex of $C$ with no adjacent edge parallel to $e_i$;
	\item $-E-\RR_{\geq 0} \cdot e_i$ with multiplicity $2$, where $E$ is an edge of $C$ parallel to $e_i$ which is bounded in direction $-e_i$.
\end{enumerate}
\end{thm}

\begin{ex}\label{ex:dualCurveExample}
In all examples, we  work over the field $\KK=\CC\{\{t^*\}\}$ of Puiseux series.
Consider the cubic plane curve
	
\[
X_1=V(x_0^2x_1+x_0x_2^2+x_0x_1x_2+tx_1^2x_2+t^2x_1x_2^2)\subset \PP^2.
\]
Its tropicalization  is the tropical curve appearing on the left side of Figure \ref{fig:dualCurveExample}, and its tropical dual is the curve on the right side of the figure. We have also included the curve $-\trop(X_1)$, since we find it easier to visually obtain $-\trop(X_1^*)$ from $\trop(X_1)$; note the sign changes in Theorem \ref{thm:dualcurve}. For convenience, every edge in the dual picture has the same color as the edge or vertex that it is dual to. Thick edges in the dual picture represent edges of multiplicity $2$.

\begin{figure}
\centering
\begin{tikzpicture}[scale=.5]

{\tiny \node [above left] at (0,0) {$(0,0)$};}
\draw[fill] (0,0) circle [radius=0.05];

\draw[color1a] (-2,-1) -- (0,0);
\draw[color2] (0,0) -- (1,1);
\draw[color3] (0,0) -- (1,0);
\draw[color4] (1,0) -- (1,1);
\draw[color5] (1,0) -- (2,-1);
\draw[color6] (1,1) -- (2,3);
\draw[color7] (2,3) -- (2,4);
\draw[] (2,3) -- (3,4);

\draw[fill, color8] (0,0) circle [radius=0.1];
\draw[fill, color9] (1,0) circle [radius=0.1];
\draw[fill, color10] (1,1) circle [radius=0.1];
\draw[fill, color11] (2,3) circle [radius=0.1];

\node at (0,-3) {$\trop(X_1)$};

\begin{scope}[shift={(10,0)}]
\draw[color1a] (-2,-1) -- (0,0);
\draw[color2, line width=.7mm] (1,1) -- (-2,-2);
\draw[color3, line width=.7mm] (0,0) -- (3,0);
\draw[color4, line width=.7mm] (1,0) -- (1,4);
\draw[color5] (1,0) -- (2,-1);
\draw[color6] (1,1) -- (2,3);
\draw[color7,line width=.7mm] (2,3) -- (2,4);

\draw[color8] (0,0) -- (0,4);
\draw[color9] (1,0) -- (-1,-2);
\draw[color10] (1,1) -- (4,1);
\draw[color11] (2,3) -- (4,3);

\node at (0,-3) {$-\trop(X_1^*)$};

\end{scope}

\begin{scope}[shift={(20,0)},xscale=-1,yscale=-1]
\draw[color1a] (-2,-1) -- (0,0);
\draw[color2, line width=.7mm] (1,1) -- (-2,-2);
\draw[color3, line width=.7mm] (0,0) -- (3,0);
\draw[color4, line width=.7mm] (1,0) -- (1,4);
\draw[color5] (1,0) -- (2,-1);
\draw[color6] (1,1) -- (2,3);
\draw[color7,line width=.7mm] (2,3) -- (2,4);

\draw[color8] (0,0) -- (0,4);
\draw[color9] (1,0) -- (-1,-2);
\draw[color10] (1,1) -- (4,1);
\draw[color11] (2,3) -- (4,3);
\node at (0,5) {$\trop(X_1^*)$};

\end{scope}
\end{tikzpicture}   

\caption{A tropical curve and its dual}
\label{fig:dualCurveExample}

\end{figure}
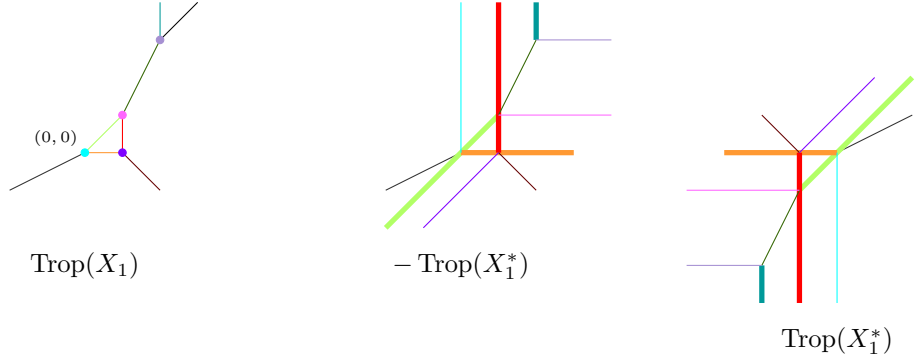

Consider instead
\[
X_2=V(x_0^2x_1+x_0x_2^2+x_0x_1x_2+t^2x_1^2x_2+tx_1x_2^2)\subset \PP^2.
\]
Its tropicalization  is the tropical curve appearing on the left side of Figure \ref{fig:dualCurveExample2}, and its tropical dual is the curve on the right side of the figure. 

Using Macaulay2 \cite{M2}, we may calculate that $X_1^*$ is cut out by the sextic
{\scriptsize{
\begin{align*}
&t^{6} {y}_{0}^{6}+(-2 t^{5}-2 t^{4}) {y}_{0}^{5} {y}_{1}+(-4 t^{6}+t^{4}+4 t^{3}+t^{2}) {y}_{0}^{4} {y}_{1}^{2}+(12 t^{4}-2 t^{2}-2 t) {y}_{0}^{3} {y}_{1}^{3}\\
&+(-12 t^{2}+1) {y}_{0}^{2} {y}_{1}^{4}+4 {y}_{0} {y}_{1}^{5}+(-2 t^{4}+4 t^{3}) {y}_{0}^{5} {y}_{2}+(4 t^{5}+4
      t^{3}-10 t^{2}) {y}_{0}^{4} {y}_{1} {y}_{2}\\
& +(8 t^{4}-26 t^{3}-2 t^{2}+8 t) {y}_{0}^{3} {y}_{1}^{2} {y}_{2}+(2 t^{2}+22 t-2) {y}_{0}^{2} {y}_{1}^{3} {y}_{2}-10 {y}_{0} {y}_{1}^{4} {y}_{2}+(2 t^{4}+t^{2}) {y}_{0}^{4} {y}_{2}^{2}\\
&+(-12 t^{3}+22 t^{2}-2 t) {y}_{0}^{3} {y}_{1}
      {y}_{2}^{2}+(-8 t^{4}-2 t^{2}-18 t+1) {y}_{0}^{2} {y}_{1}^{2} {y}_{2}^{2}+(-20 t^{2}+8) {y}_{0} {y}_{1}^{3} {y}_{2}^{2}+{y}_{1}^{4} {y}_{2}^{2}\\
&-2 t^{2} {y}_{0}^{3} {y}_{2}^{3}+(8 t^{3}+8 t) {y}_{0}^{2} {y}_{1} {y}_{2}^{3}+(8 t^{2}+22 t-2) {y}_{0} {y}_{1}^{2} {y}_{2}^{3}-2
      {y}_{1}^{3} {y}_{2}^{3}+t^{2} {y}_{0}^{2} {y}_{2}^{4}\\
 &-10 t {y}_{0} {y}_{1} {y}_{2}^{4}+(-4 t^{2}+1) {y}_{1}^{2} {y}_{2}^{4}+4 t {y}_{1} {y}_{2}^{5}
\end{align*}}}
while $X_2^*$ is cut out by the sextic
{\scriptsize{
\begin{align*}
&t^{6} {y}_{0}^{6}+(-2 t^{5}-2 t^{4}) {y}_{0}^{5} {y}_{1}+(t^{4}+t^{2}) {y}_{0}^{4} {y}_{1}^{2}+(10 t^{2}-2 t) {y}_{0}^{3} {y}_{1}^{3}+(-12 t+1) {y}_{0}^{2} {y}_{1}^{4}+4 {y}_{0} {y}_{1}^{5}\\
&+(4 t^{6}-2 t^{5}) {y}_{0}^{5} {y}_{2}+(-6 t^{4}+4 t^{3}) {y}_{0}^{4} {y}_{1}
      {y}_{2}+(-26 t^{3}+16 t^{2}-2 t) {y}_{0}^{3} {y}_{1}^{2} {y}_{2}\\
&+(22 t^{2}+2 t-2) {y}_{0}^{2} {y}_{1}^{3} {y}_{2}-10 {y}_{0} {y}_{1}^{4} {y}_{2}+(2 t^{5}+t^{4}) {y}_{0}^{4} {y}_{2}^{2}+(22 t^{4}-12 t^{3}-2 t^{2}) {y}_{0}^{3} {y}_{1} {y}_{2}^{2}\\
&+(-26 t^{2}-2 t+1){y}_{0}^{2} {y}_{1}^{2} {y}_{2}^{2}+(-20 t+8) {y}_{0} {y}_{1}^{3} {y}_{2}^{2}+{y}_{1}^{4} {y}_{2}^{2}-2 t^{4} {y}_{0}^{3} {y}_{2}^{3}\\
&+(8 t^{3}+8 t^{2}) {y}_{0}^{2} {y}_{1} {y}_{2}^{3}+(22 t^{2}+8 t-2) {y}_{0} {y}_{1}^{2} {y}_{2}^{3}-2 {y}_{1}^{3} {y}_{2}^{3}+t^{4}
{y}_{0}^{2} {y}_{2}^{4}\\
&-10 t^{2} {y}_{0} {y}_{1} {y}_{2}^{4}+(-4 t+1) {y}_{1}^{2} {y}_{2}^{4}+4 t^{2} {y}_{1} {y}_{2}^{5}.
      \end{align*}}}
Tropicalizing these two curves, we obtain the same result as provided above by the theorem.
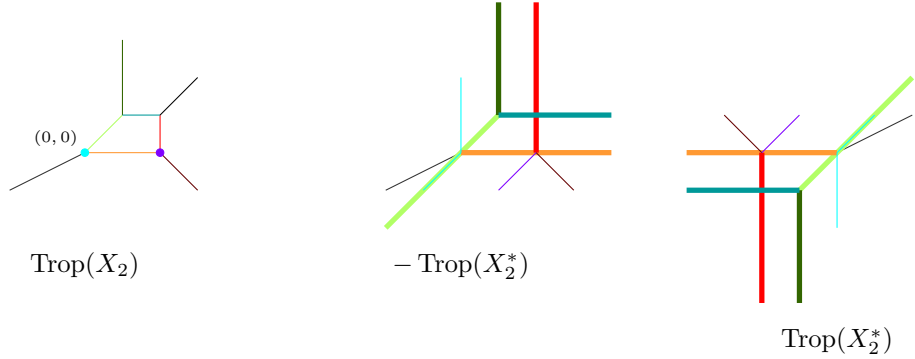
\begin{figure}
\centering
\begin{tikzpicture}[scale=.5]

{\tiny \node [above left] at (0,0) {$(0,0)$};}
\draw[fill] (0,0) circle [radius=0.05];

\draw[color1a] (-2,-1) -- (0,0);
\draw[color2] (0,0) -- (1,1);
\draw[color3] (0,0) -- (2,0);
\draw[color4] (2,0) -- (2,1);
\draw[color5] (2,0) -- (3,-1);
\draw[color6] (1,1) -- (1,3);
\draw[color7] (1,1) -- (2,1);
\draw[] (2,1) -- (3,2);

\draw[fill, color8] (0,0) circle [radius=0.1];
\draw[fill, color9] (2,0) circle [radius=0.1];

\node at (0,-3) {$\trop(X_2)$};

\begin{scope}[shift={(10,0)}]
\draw[color1a] (-2,-1) -- (0,0);
\draw[color2, line width=.7mm] (1,1) -- (-2,-2);
\draw[color3, line width=.7mm] (0,0) -- (4,0);
\draw[color4, line width=.7mm] (2,0) -- (2,4);
\draw[color5] (2,0) -- (3,-1);
\draw[color6, line width=.7mm] (1,1) -- (1,4);
\draw[color7, line width=.7mm] (1,1) -- (4,1);

\draw[color8] (0,0) -- (0,2);
\draw[color8] (0,0) -- (-1,-1);
\draw[color9] (2,0) -- (1,-1);
\node at (0,-3) {$-\trop(X_2^*)$};
\end{scope}

\begin{scope}[shift={(20,0)}, xscale=-1,yscale=-1]
\draw[color1a] (-2,-1) -- (0,0);
\draw[color2, line width=.7mm] (1,1) -- (-2,-2);
\draw[color3, line width=.7mm] (0,0) -- (4,0);
\draw[color4, line width=.7mm] (2,0) -- (2,4);
\draw[color5] (2,0) -- (3,-1);
\draw[color6, line width=.7mm] (1,1) -- (1,4);
\draw[color7, line width=.7mm] (1,1) -- (4,1);

\draw[color8] (0,0) -- (0,2);
\draw[color8] (0,0) -- (-1,-1);
\draw[color9] (2,0) -- (1,-1);
\node at (0,5) {$\trop(X_2^*)$};

\end{scope}\end{tikzpicture}   

\caption{A tropical curve and its dual}
\label{fig:dualCurveExample2}

\end{figure}
\end{ex}

While Theorem \ref{thm:dualcurve} gives a complete description of $\trop(X^*)$ in the curve case, it is interesting to know a bit more: what are the tropical tangents for any given point of $C$?
The answer is found in  \S \ref{sec:surfacesintro}, see Theorems \ref{thm:vert} and \ref{thm:edge}.

Rather than asking for the tropicalization of $X^*$, we could ask for less refined information such as the Newton polygon:
\begin{cor}[See \S\ref{sec:curves}]\label{cor:newton}
Let $\Delta_X\subset\RR^2$ be the Newton polygon of $X$, and assume that $X$ is sufficiently generic with respect to $\Delta_X$. This is in particular satisfied if $\trop(X)$ is smooth.

Label the vectors of $\Delta_X$ by $v_1,v_2,\ldots,v_m$ in counterclockwise orientation, omitting all edges parallel to $w_0=(-1,1),w_1=(0,-1),w_2=(1,0)$. For $i=0,1,2$, let $\sigma_i$ be the sum of edge vectors parallel to $w_i$. Then the edge vectors of the Newton polygon $\Delta_{X^*}$ of $X^*$ are exactly 
\[
-v_m,-v_{m-1},\ldots,-v_1
\]
along with 
\[
\vol(\Delta_X)\cdot w_i-\sigma_i
\]
where $\vol(\Delta_X)$ is the normalized lattice volume of $\Delta_X$.
\end{cor}
\noindent Here, the condition that $X$ is generic with respect to $\Delta_X$  means that $X$ is in a certain non-empty Zariski open subset of the space of all hypersurfaces with Newton polygon $\Delta_X$.

\begin{ex}
The Newton polygon of the curve $X_1$ from Example \ref{ex:dualCurveExample} is displayed on the left of Figure \ref{fig:NewtonPolygon}.  
Following the recipe above, or taking the Newton polygon associated to $\trop(X_1^*)$, we find that the Newton polygon of $X_1^*$ is the one on the right side of the figure.
\begin{figure}[h]
\centering
\begin{tikzpicture}[scale=.7]
\draw (1,0)--(2,1)--(1,2)--(0,2)--(1,0);
{\scriptsize
\node [below] at (1,0) {$(1,0)$};
\node [right] at (2,1) {$(2,1)$};
\node [above] at (1,2) {$(1,2)$};
\node [left] at (0,2) {$(0,2)$};
}
\begin{scope}[shift={(5,-1)}]
\draw (0,0) -- (5,0) -- (4,2) -- (1,5) -- (0,4) -- (0,0);

{\scriptsize
\node [below] at (0,0) {$(0,0)$};
\node [below] at (5,0) {$(5,0)$};
\node [right] at (4,2) {$(4,2)$};
\node [above] at (1,5) {$(1,5)$};
\node [left] at (0,4) {$(0,4)$};
}

\end{scope}
\end{tikzpicture}
\caption{The Newton polygons of $X_1$ and $X_1^*$.}
\label{fig:NewtonPolygon}
\end{figure}
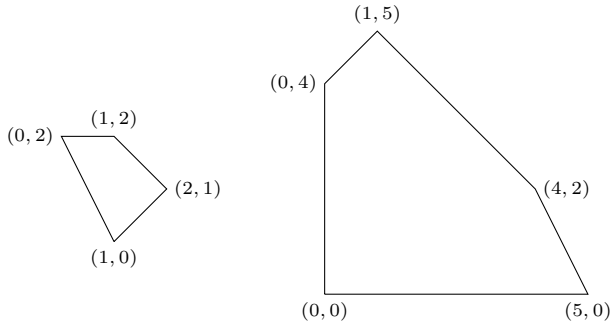
Note that even though $\trop(X_2)$ is quite different from $\trop(X_1)$, their duals have the same Newton polygons. This is to be expected since $X_1$ and $X_2$ have the same Newton polygon and have smooth tropicalizations.
\end{ex}

Using Corollary \ref{cor:newton}, we also obtain a formula for the degree of $X^*$ whenever $X$ has   singularities which are sufficiently generic with respect to its Newton polygon. We may then use this to recover Pl\"ucker's formula for the degree of the dual of a curve with nodal and cuspidal singularities \eqref{eqn:pluecker}, and more generally  to partially recover Teissier's formula for the dual degree, see Remark \ref{rem:df}.

While standard projective duality is an involution on varieties, this cannot be true in the tropical setting. However, we do show that by enhancing the dual tropical curve $\trop(X^*)$ with a bit of extra information, the original tropical curve $\trop(X)$ may be reconstructed,
as long as $\trop(X)$ was smooth. See Proposition \ref{prop:injectivity}.

\subsection{Results: surfaces and beyond}\label{sec:surfacesintro}
We now consider a smooth tropical hypersurface $\trop(X)\subset \RR^n$, where $X=V(f)\subset \PP^n$ is a hypersurface of degree at least two. As in the curve case, $\trop(X)$ comes equipped with the structure of a polyhedral complex induced by the subdivision of the Newton polytope of $f$. As above, zero-dimensional cells will be called vertices, and one-dimensional cells (possibly unbounded) will be called edges.
We let $e_1,\ldots,e_n$ be the standard basis for $\RR^n$, with $e_0=-e_1-e_2-\ldots-e_n$. We are using the standard coordinates $x_1/x_0,\ldots,x_n/x_0$ on the torus $(\KK^*)^n$.
For any subset $J\subset \{0,\ldots,n\}$, let $\langle J \rangle$ be the span of $\{e_i\}_{i\in J}$. For any subset $S\subset \RR^n$, we denote by $\langle S\rangle$ the subspace generated by differences of elements of $S$.

In general, a description of the tropical tangents for vertices of $\trop(X)$ is straightforward. Fix a vertex $p\in\trop(X)$. Similarly to the case of curves, the tropical tangents to $p$ are shifted in precisely the standard directions that don't coincide with directions of edges emanating from $p$. More precisely,  let 
$\cJ(p)$ be the collection of maximal sets $J\subset \{0,\ldots,n\}$ such that 
 for every edge $E$ adjacent to $p$,
\[
\langle E \rangle \cap \langle J \rangle = 0.
\]

\begin{thm}[See \S \ref{sec:vert}]\label{thm:vert} Let $p\in \trop(X)$ be a vertex. Then the closure of the set of all tropical tangents to $p$ is the union over all $J\in\cJ(p)$ of 
	\[
		-p-\sum_{i\in J}\RR_{\geq 0} \cdot e_i.
	\]
\end{thm}

\begin{ex}[Quadric surface]\label{ex:surface}
We consider the smooth quadric surface 
\[X=V(tx_0^2+tx_0x_1+tx_0x_2+tx_0x_3+x_1x_2).\]
Its tropicalization has exactly two vertices: 
\[
	v_1=-e_1 \qquad\mathrm{and}\ v_2=-e_2.
\]
There are a total of seven edges (see Figure \ref{fig:tropicalSurface}):
\begin{align*}
&E_0=\overline{v_1v_2}\\
&E_{i1}=v_i-\RR_{\geq 0}\cdot e_3\qquad &i=1,2\\
&E_{i2}=v_i-\RR_{\geq 0}\cdot e_i\qquad &i=1,2\\
&E_{i3}=v_i-\RR_{\geq 0}\cdot (e_0+e_i)\qquad &i=1,2\\
\end{align*}
There are nine two-faces, which are formed by the following convex hulls of edges:
\begin{align*}
&F_j: E_0,E_{1j},E_{2j} \qquad & j=1,2,3\\
&F_{ijk}: E_{ij},E_{ik} \qquad & i=1,2,\ j,k=1,2,3,\ j\neq k. 
\end{align*}

\begin{figure}[htbp]
\begin{center}
\includegraphics[width=0.5\textwidth]{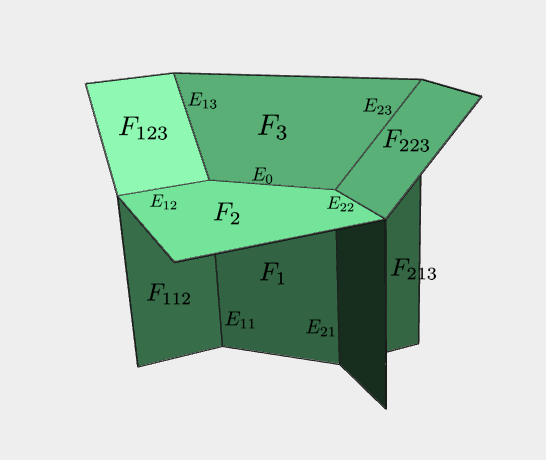}
\caption{The tropical surface $\trop(X)$.}
\label{fig:tropicalSurface}
\end{center}
\end{figure}

We use Theorem \ref{thm:vert} to determine the tropical tangents to the vertices. For $p=v_i,v_2$, $\cJ(p)$ consists of the  single set $\{0,2\}$ or $\{0,1\}$, respectively. By the theorem, we thus obtain the tropical tangents
\begin{align*}
	&-v_1-\RR_{\geq 0}\cdot e_0-\RR_{\geq 0}\cdot e_2\\
	&-v_2-\RR_{\geq 0}\cdot e_0-\RR_{\geq 0}\cdot e_1.
\end{align*}
\end{ex}

For tropical tangents of edges and higher dimensional cells in $\trop(X)$, we need a bit more notation.
For any rational polyhedron $P$ with facet $Q$, let $v_Q$ denote the primitive inward normal vector, and define the function:
\begin{align*}
\dd_Q:P&\to \RR\\
p&\mapsto \langle p-q,v_Q\rangle,
\end{align*}
where $q$ is any point of $Q$.
The quantity $\dd_Q(p)$ is sometimes called the \emph{lattice distance} from $p$ to the facet $Q$.
The function $\delta(p)=\min_Q \delta_Q(p)$ is a piecewise linear, concave function on $P$.
If $P$ has no facets (that is, $P$ equals its affine span), then $\dd(p)=\infty$.
If $P$ is any polyhedron contained in $\trop(X)$, we let $P^*$ be the Euclidean closure of all tropical tangents which are tangent at some point $p$ in the relative interior of $P$.

Returning to edges, 
fix an edge $E\subset \trop(X)$. The tropical tangents to $E$ will be shifted in standard directions that are distinct from the directions of the two-faces emanating from $E$. The shift will be unrestricted in directions that are not already contained in $E$. More precisely,
Let $\cJ(E)$ be the collection of  maximal sets $J\subset \{0,\ldots,n\}$ such that 
 for every two-face $F$ adjacent to $E$,
\[
\langle F \rangle \cap \langle J \rangle = \langle E \rangle \cap \langle J \rangle.
\]
For each $J\in \cJ(E)$, let $J'$ be the minimal subset of $J$ such that $\langle E \rangle \subset \langle J' \rangle$ if such a subset exists, and $\emptyset$ otherwise.

\begin{thm}[See \S \ref{sec:edges}]\label{thm:edge}
Let $p$ be a point in the interior of an edge $E\subset \trop(X)$. Then the tropical tangents of $p$ are the union
over all $J\in \cJ(E)$ of
\[
		-p-\sum_{j\in J}s_j e_j,
	\]
with $s_j\in \RR_{\geq 0}$, and the values $\{s_i\}_{i\in{J'}}$ satisfy either
\begin{enumerate}
\item the minimum is obtained at least twice and is at most $\delta(p)$; or
\item the  minimum equals $\delta(p)$; or
\item $p$ is the the midpoint of $E$, and the minimum is at least $\delta(p)$;
%
\end{enumerate}
There is no condition on the $s_j$ with $j\in J\setminus J'$.
\end{thm}

\begin{ex}[Quadric surface continued]
	We continue Example \ref{ex:surface}. First we consider the edge $E_0$, which is parallel to $e_1-e_2$. Here, $\cJ(E_0)$ consists only of $\{0\}$, with corresponding $J'=\emptyset$. This contributes the tropical tangents $-E_0-\RR_{\geq 0}\cdot e_0$.

	Next we consider an edge $E_{i1}$, which is parallel to $e_3$. Here we obtain that $\cJ(E_{i1})$ consists only of $\{0,3\}$, with $J'=\{3\}$. For a point $p=v_i-\lambda\cdot e_3$, its lattice distance is $\dd(p)=\lambda$, and we obtain the tropical tangents
	\[
		\{-v_i+\lambda e_3-\lambda e_3\}-\RR_{\geq 0}\cdot e_0=-v_i-\RR_{\geq 0}\cdot e_0.
	\]
In other words, the tangents coming from the edge $E_{i1}$ form a two dimensional cone in the dual space spanned by $-e_0$ and $-e_3$.

	Consider instead an edge $E_{i2}$, which is parallel to $e_i$. Here, $\cJ(E_{i2})$ consists only of $\{i\}$, with $J'=\{i\}$. For $p=v_i-\lambda\cdot e_i$ we have $\dd(p)=\lambda$, and we only obtain the tropical tangent $-v_i$. 

	Finally, consider an edge $E_{i3}$, which is parallel to $-e_0-e_i$. Here, the only element of $\cJ(E_{i2})$ is $\emptyset$, and we obtain $-E_{i3}$ in $\trop(X^*)$. 
\end{ex}

As we increase the dimension of the strata of $\trop(X)$ where we look for tropical tangents, the complexity increases. We will now assume that we are in the situation $n=3$,  $\trop(X)$ is a tropical surface, and we have fixed a two-face $F\subset \trop(X)$. 
We will also often be assuming that $X$ has \emph{generic valuations}, see Definition \ref{defn:gen}.
Roughly speaking, this requires the coefficients of $f$ to be chosen such that their valuations are sufficiently generic.
Many situations, including those of bounded faces, are then particularly straightforward:

\begin{thm}[See \S \ref{face:4}]\label{thm:face}
Let $X$ be a smooth tropical surface with generic valuations.
Let $J\subset \{0,1,2,3\}$ consist of those $i$ such that $e_i\in\langle F \rangle$. Assume that the recession cone of $F$ does not intersect $\sum_{i\in J} \RR_{\geq 0} \cdot (-e_i)$ non-trivially. Then 
\[
F^*=-F-\sum_{i\in J} \RR_{\geq 0} \cdot e_i.
\]
\end{thm}

To get more precise information about tropical tangents, and to deal with those situations not covered by the above theorem, we will split things up into a number of cases.

\begin{prop}[See \S \ref{face:1}]\label{prop:face}
Let $F$ be a two-face of the smooth tropical surface $\trop(X)$, such that $e_i\notin \langle F\rangle $ for all $i$. Then for any $p$ in the relative interior of $F$, the only tropical tangent to $p$ is $-p$. 
\end{prop}
\begin{ex}[Quadric surface continued]
	We continue Example \ref{ex:surface}. 
	The only face of $\trop(X)$ to which Proposition \ref{prop:face} applies is $F_3$. We thus obtain $-F_3$ as tropical tangents.
\end{ex}

We next deal with faces that contain a single standard direction. 
Consider the piecewise linear concave function $\dd:p\mapsto \min_{E} \dd_E(p)$, where $E$ are the edges of $F$.
This induces a polyhedral subdivision $\cS_F$ of $F$. 
For any $S\in \cS_F$, let $E_S$ denote the  edges $E$ of $F$ for which $\dd_E$ is minimal on $S$. 
In the following propositions, we will always assume that $\trop(X)$ has generic valuations.
Note that in this case, $E_S$ consists of $3-\dim(S)$ edges.

\begin{prop}[See \S \ref{face:2}]\label{prop:face2}
	Let $F$ be a face of the smooth tropical surface $\trop(X)$ such that $e_i\in \langle F \rangle$, but $e_j\notin \langle F \rangle$ for $j\neq i$. Fix $S\in \cS_F$ intersecting the relative interior of $F$. Assume that $\trop(X)$ has generic valuations. We distinguish several (not mutually exclusive) cases.

\begin{enumerate}
	\item Not every $E\in E_S$ is parallel to $e_i$. For $p$ in the relative interior of $S$, the tropical tangents  are
\[-p-s_i e_i,\]
where $s_i=\dd(p)$ if exactly one element of $E_S$ is not parallel to $e_i$, and $s_i\geq\dd(p)$ otherwise. 

\item Every $E\in E_S$ is parallel to $e_i$. If  $-e_i$ is not in the recession cone of $S$, then \[
		S^*=\{-p-\dd(p)e_i\ |\ p\in S\}-\RR_{\geq 0}\cdot e_i.\]
	\item If $-e_i$ is in the recession cone of $F$, then 
	\[
	F^*=\bigcup_{E} -E \cup \bigcup_{v} \left(-v-\RR_{\geq 0}e_i\right),
	\]
	where  $E$ varies over edges of $F$ not parallel to $e_i$, and  $v$ varies over vertices of $F$ not adjacent to edges parallel to $e_i$.
\end{enumerate}	
\end{prop}

\begin{ex}[Quadric surface continued]
	We continue Example \ref{ex:surface}. 
	Face $F_1$ falls under the third case of Proposition  \ref{prop:face2}, and we obtain $-E_0$ as its tropical tangents. 
\end{ex}

The most tedious case is when $\langle F \rangle=\langle e_i,e_j\rangle$. We split this up into several propositions:

\begin{prop}[See \S \ref{face:3}]\label{prop:face3}
Let $F$ be a face of the smooth tropical surface $\trop(X)$ such that $\langle F \rangle=\langle e_i,e_j\rangle$,  $i\neq j$. Fix $S\in \cS_F$ intersecting the relative interior of $F$. 
	Assume that $\trop(X)$ has generic valuations.
 
Assume that either $\dim S=0$, $\dim S=1$ and the edges in $E_S$ are not parallel, or $\dim S=2$ and the single edge in $E_S$ is not parallel to $e_i$ or $e_j$. Then the closure of the set of tropical tangents for $p$ in the interior of $S$ is 
		\begin{align*}
			-p-\dd(p)(e_i+e_j)-\lambda_ie_i-\lambda_je_j,
		\end{align*}
		subject to the additional conditions on $\lambda_i,\lambda_j\in\RR_{\geq 0}$:
		\begin{align*}
			\lambda_i\lambda_j=0&\qquad&\textrm{if edges in $E_S$ are  in exactly two directions}\\
			\lambda_i=0&\qquad&\textrm{if all but exactly one edge in $E_S$ are parallel to $e_i$}\\
			\lambda_j=0&\qquad&\textrm{if all but exactly one edge in $E_S$ are parallel to $e_j$}.
		\end{align*}

\noindent In particular, $\lambda_i=\lambda_j=0$ whenever $\dim(S)=2$, and we always have $\lambda_i\lambda_j =0$ when  $\dim(S)=1$ (where, in addition, $\lambda_i$ or $\lambda_j$ is $0$ if  an edge in $E_S$ is parallel to $e_i$ or $e_j$ respectively).
\end{prop}

\begin{prop}[See \S \ref{face:3}]\label{prop:face5}
Let $F$ be a face of the smooth tropical surface $\trop(X)$ such that $\langle F \rangle=\langle e_i,e_j\rangle$,  $i\neq j$. Fix $S\in \cS_F$ intersecting the relative interior of $F$. 
	Assume that $\trop(X)$ has generic valuations.
 
Assume that $\dim S=2$ and the edge in $E_S$ is parallel to $e_j$.

\begin{enumerate}
\item If $-e_j$ is not in the recession cone of $F$, then 

\[S^*=\{-p-\dd(p)(e_i+e_j)\ |\ p\in S\}-\RR_{\geq 0}\cdot e_j.
		\]

	\item  If $-e_j$ is in the recession cone of $F$,
then 
\[
	S^*=\{-p-\dd(p)(e_i+e_j)-\lambda_je_j\ |\ p\in S'\}
\]
where $S'$ consists of the edges of $S$ not parallel to $e_j$, and $\lambda_j\geq 0$ is equal to zero except at the interior vertices of $S'$.
\end{enumerate}
\end{prop}

The remaining case is where $S\in \cS_F$ is an edge, and $E_S$ consists of two parallel edges. For this, we need additional notation. 
If $S$ is a ray, we may write it as $S=q+\RR_{\geq 0}\cdot v$ for some unique primitive lattice vector $v$. Let $E$ be the unique element of $E_q\setminus E_S$. Then $S$ is \emph{purely primitive} if $\dd_E(q+v)-\dd_E(q)=1$.

\begin{prop}[See \S \ref{face:3}]\label{prop:face4}
Let $F$ be a face of the smooth tropical surface $\trop(X)$ such that $\langle F \rangle=\langle e_i,e_j\rangle$,  $i\neq j$. Fix $S\in \cS_F$ intersecting the relative interior of $F$. 
	Assume that $\trop(X)$ has generic valuations.
 
Assume that $\dim S=1$ and the edges in $E_S$ are parallel.
Then the set $S^*$ consists of all 
		\begin{align*}
			-p-\dd(p)(e_i+e_j)-\lambda_ie_i-\lambda_je_j,
		\end{align*}
for $p\in S$, $\lambda_i,\lambda_j\in\RR_{\geq 0}$ subject to the following conditions:
\begin{enumerate}
\item If $S$ is bounded: no conditions.
\item If $S$ is a line: $\lambda_i=\lambda_j$.
\item If $S$ is a ray of the form $q+\RR_{\geq 0}\cdot (\alpha_ie_i+\alpha_je_j)$:

\vspace{.2cm}

\begin{center}
\begin{tabular}{|l|l|l|l|}
\hline
$\alpha_i$ & $\alpha_j$ & $S$ purely primitive & conditions\\
\hline
\multirow{2}{*}{$-1$}&
\multirow{2}{*}{$-1$}& yes & $\lambda_i=\lambda_j=0$ unless $p=q$, and $\lambda_i\lambda_j=0$\\
&& no & $\lambda_i=\lambda_j=0$ unless $p=q$\\
\hline
\multirow{2}{*}{$-1$}&
\multirow{2}{*}{$0$}& 
 yes & $p=q$ and $\lambda_i=0$\\
&& no & $p=q$\\
\hline
\multirow{2}{*}{$0$}&
\multirow{2}{*}{$-1$}& 
 yes & $p=q$ and $\lambda_j=0$\\
&& no & $p=q$\\
\hline
$-1$ & $<-1$ & no & $\lambda_j=0$ unless $p=q$\\
\hline
$<-1$ & $-1$ & no & $\lambda_i=0$ unless $p=q$\\
\hline
\multicolumn{3}{|l|}{all other cases}& $\lambda_i\lambda_j=0$\\
\hline
\end{tabular}
\end{center}
\vspace{.2cm}

\end{enumerate}

\end{prop}

\begin{ex}[Quadric surface continued]
	We continue Example \ref{ex:surface}. 
	The faces we will consider are $F_2$, and $F_{i12}$, $F_{i13}$, $F_{i23}$ for $i=1,2$.

	We begin with $F_2$. The subdivision $\cS_{F_2}$ is pictured in Figure \ref{fig:face}, as are the contributions of each strata $S$ to the tropical tangents (where the colour of a cell in $\cS_{F_2}$ matches the colour of the corresponding cell in the dual picture). These are obtained by applying Propositions \ref{prop:face3} and \ref{prop:face5}. Combining them, we obtain 
	\[
		\conv\{e_1,e_2\}-\RR_{\geq 0}\cdot e_1-\RR_{\geq0}\cdot  e_2.
	\]

	A similar calculation shows that for $F_{i12}$, we only obtain $-v_i$ as a tropical tangent. For $F_{i13}$, we obtain
	\[
		-v_i-\RR_{\geq 0}\cdot (-e_0-e_i)-\RR_{\geq 0}\cdot (-e_0-e_i-e_3).
	\]
Finally, for $F_{i23}$ we obtain
\[
-v_i-\RR_{\geq 0}\cdot (-e_0-e_i)-\RR_{\geq 0}\cdot e_0.
\]
\end{ex}

\begin{figure}
	\begin{tikzpicture}
\begin{scope}
	\draw [color1,fill=color1] (-1,0) -- (-4,0) -- (-4,-4) -- (-1,-1) -- (-1,0);
	\draw [color2,fill=color2] (0,-1) -- (0,-4) -- (-4,-4) -- (-1,-1) -- (0,-1);
	\draw [color3,fill=color3] (-1,0) -- (0,-1) -- (-1,-1) -- (-1,0);
	\draw [line width=.5mm, color4] (0,-1) -- (-1,-1);
	\draw [line width=.5mm, color5] (-1,0) -- (-1,-1);
	\draw [line width=.5mm, color6] (-4,-4) -- (-1,-1);
	\draw[fill, black] (-1,0) circle [radius=0.05];
	\draw[fill, black] (0,-1) circle [radius=0.05];
	\draw[fill, color7] (-1,-1) circle [radius=0.1];
	\node [right] at (0,-1) {$-e_2$};
	\node [above] at (-1,0) {$-e_1$};
	\node at (-2,-4.5) {$\cS_F$};
\end{scope}
\begin{scope}[shift={(5,-2)},xscale=-1,yscale=-1]
	\draw [color3,fill=color3] (-1,0) -- (0,-1) -- (0,0) -- (-1,0);
	\draw [color1, line width=1mm] (2,0) -- (-1,0);
	\draw [color2, line width=1mm] (0,2) -- (0,-1);
	\draw[fill, black] (-1,0) circle [radius=0.05];
	\draw[fill, black] (0,0) circle [radius=0.05];
	\draw[fill, black] (0,-1) circle [radius=0.05];
	\node [above left]  at (0,0) {$0$};
	\node [right] at (0,-1) {$e_2$};
	\node [above right] at (-1,0) {$e_1$};
	\node at (0,2.5) {Tangents for $\dim S=2$};

\end{scope}

\begin{scope}[shift={(-2,-6.5)},xscale=-1,yscale=-1]
	\draw [fill, color4] (0,-1) -- (0,0) -- (2,0) -- (2,-1) -- (0,-1);
	\draw [fill, color5] (-1,0) -- (0,0) -- (0,2) -- (-1,2) -- (-1,0);
	\draw [line width=1mm, color6] (2,0) -- (0,0) -- (0,2);
	\draw[fill, black] (-1,0) circle [radius=0.05];
	\draw[fill, black] (0,0) circle [radius=0.05];
	\draw[fill, black] (0,-1) circle [radius=0.05];
	\node [above right] at (0,0) {$0$};
	\node [right] at (0,-1) {$e_2$};
	\node [above] at (-1,0) {$e_1$};
	\node at (0,3) {Tangents for $\dim S=1$};
\end{scope}
\begin{scope}[shift={(5,-6.5)},xscale=-1,yscale=-1]
	\draw [fill, color7] (0,0) -- (2,0) -- (2,2) -- (0,2) --  (0,0);
	\draw[fill, black] (0,0) circle [radius=0.05];
	\node [above right] at (0,0) {$0$};
	\node at (0,3) {Tangents for $\dim S=0$};
\end{scope}
	\end{tikzpicture}
	\caption{The subdivision $\cS_{F_2}$ and tropical tangents for the face $F_2$}\label{fig:face}
	
\end{figure}
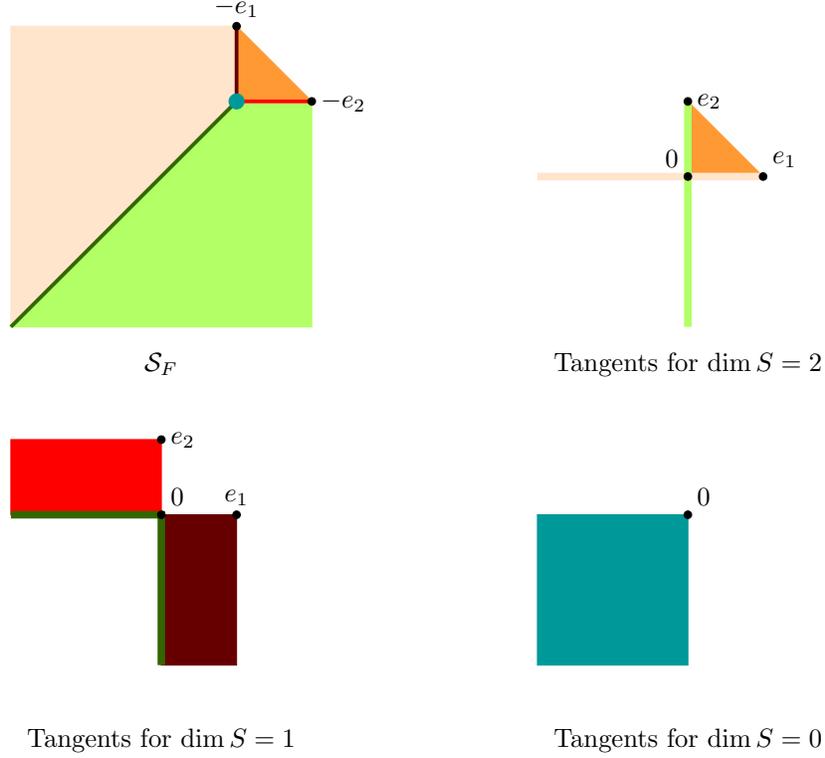

Putting this all together, we are able to describe the tropical variety $\trop(X^*)$ when $X$ is a surface:
\begin{thm}[See \S \ref{face:4}]\label{thm:surface}
	Let $X\subset \PP^3$ be a surface with $\trop(X)$ smooth with generic valuations. Then $\trop(X^*)$ is the union of the polyhedra in $\RR^3$ resulting from Theorems \ref{thm:vert} and \ref{thm:edge} and Propositions \ref{prop:face}, \ref{prop:face2}, \ref{prop:face3}, \ref{prop:face5}, and \ref{prop:face4}.
\end{thm}
\begin{ex}[Quadric surface continued]\label{ex:surfacea}
	We continue Example \ref{ex:surface}. 
Putting together the tropical tangents we have determined, we conclude that $\trop(X^*)$ has vertices $v_1'=e_1$ and $v_2'=e_2$, edges
\begin{align*}
&E_0'=\overline{v_1'v_2'}\\
&E_{i1}'=v_i'-\RR_{\geq 0}\cdot e_0\qquad &i=1,2\\
&E_{i2}'=v_i'-\RR_{\geq 0}\cdot e_j\qquad &i,j=1,2,\ j\neq i\\
&E_{i3}'=v_i'+\RR_{\geq 0}\cdot (e_0+e_i)\qquad &i=1,2\\
\end{align*}
There are nine two-faces, which are formed by the following convex hulls of edges:
\begin{align*}
&F_j': E_0',E_{1j}',E_{2j}' \qquad & j=1,2,3\\
&F_{ijk}': E_{ij}',E_{ik}' \qquad & i=1,2,\ j,k=1,2,3,\ j\neq k. 
\end{align*}
The contribution of each face of $\trop(X)$ is summarized in the following table:

\begin{align*}
\begin{array}{l @{\qquad\rightsquigarrow\qquad} l}
v_i & F_{i12}'\\
E_0 & F_1'\\
E_{i1} & E_{i1}'\\
E_{i2} & v_i'\\
E_{i3} & E_{i3}'\\
\end{array}
\qquad\qquad
\begin{array}{l @{\qquad\rightsquigarrow\qquad} l}
F_{1} & E_0'\\
F_{2} & F_{2}'\\
F_{3} & F_{3}'\\
F_{i12} & v_i'\\
F_{i13} & F_{i23}'\\
F_{i23} & F_{i13}'\\
\end{array}.
\end{align*}

Alternatively, one could compute using Macaulay2 \cite{M2} that $X^*$ is cut out by 
\[
ty_1y_2-ty_1y_3-ty_2y_3+y_0y_3+(t-1)y_3^2.
\]
The tropicalization of this variety agrees with the description of $\trop(X^*)$ above.
\end{ex}

\begin{rem}[Dual Defects]
Given an irreducible variety $X\subset \PP^n$, it is an important problem to determine the dimension of $X^*$. This is measured by the \emph{dual defect}:
\[
\defect(X)=n-\dim X^*-1.
\]
Typically, the defect is zero, that is, $X^*$ is a hypersurface. In fact, if $\defect(X)=k$, then $X$ is covered by $k$-planes \cite[\S 1]{tevelev}.

Any plane curve of degree larger than one will automatically have defect zero. The situation for surfaces in $\PP^3$ is more interesting. Since the dimesion of $\trop(X^*)$ equals that of $X^*$, we may use our Theorem \ref{thm:surface} to determine the dual defect of any surface $X\subset \PP^3$ for which $\trop(X)$ is smooth and has generic valuations. We illustrate this with two examples.
\end{rem}
\begin{ex}[Smooth quadric]\label{ex:smoothq}
The smooth quadric surface $X$ from Example \ref{ex:surface} is isomorphic to $\PP^1\times \PP^1$, and is ruled by lines, so it has the potential for being defective. However, we have seen in Example \ref{ex:surfacea} that its dual variety is also a surface, so it has defect zero.

Alternatively, after a change of coordinates, we might consider the smooth quadric $X=V(x_0x_1-x_2x_3)\subset\PP^3$.
The tropical variety $\trop(X)$ is the linear subspace $F$ of $\RR^3$ spanned by 
$e_1+e_2$ and $e_1+e_3$. This face $F$ satisfies $e_i\notin\langle F \rangle$ for all $i$, so Proposition \ref{prop:face} applies, and we conclude that $\trop(X^*)=-F$. In particular, $\dim \trop(X^*)=2$, so $\defect(X)=0$.

Note that  $X$ is an example of a toric variety. In \cite{sturmfels}, A.~Dickenstein, E.~Feichtner, and B.~Sturmfels show how to describe $\trop(X^*)$ for any toric variety.
\end{ex}
\begin{ex}[Quadric cone]\label{ex:cone}
Consider the singular quadric \[X=V(x_0^2+x_1x_2+x_1x_3)\] in $\PP^3$. Then $\trop(X)$ is the fan with lineality space $E$ generated by $e_1-e_2-e_3$, and three two-dimensional cones generated by $-e_2$, $-e_3$, and $-e_0$. 
It is smooth with generic valuations.

For the unique edge $E$ of $\trop(X)$, Theorem \ref{thm:edge} yields $-E$ in the tropical dual. For each two-face, the first case of Proposition \ref{prop:face2} applies, yielding $-E$ in the tropical dual: if $p\in F$ is of the form $q-\lambda e_i$ for some $\lambda>0$, then $\dd(p)=\lambda$, and $-p-\dd(p)e_i=-q$.
We conclude that $\dim(X^*)=1$, so $\defect(X)=1$. In fact, $X$ is a cone over a singular quadric, so $X^*$ is properly contained in a hyperplane (and thus defective).
\end{ex}

\subsection{Related work}\label{sec:related}
We have already mentioned the work of Z.~Izhakian \cite[\S 2.1]{izhakian}, which shows that for a hypersurface $X=V(f)$ with the lowest valuation parts of the coefficients of $f$ sufficiently generic, $\trop(X^*)$ may be determined from the tropicalization of $f$, although this is not made explicit except for the case of quadric curves. Groundbreaking work on tropical dual varieties was done by A.~Dickenstein, E.~Feichtner, and B.~Sturmfels, in which $\trop(X^*)$ was explicitly described for any projective \emph{toric} variety $X$, not limited to the case of $X$ being a hypersurface \cite{sturmfels}. 

In a slightly different setting, B.~Bertrand, E.~Brugall\'e, and G.~Mikhalkin study \emph{pretangencies} of tropical cycles in the plane to tropical morphisms \cite{brugalle}. \emph{Bitangents} of tropical plane curves have been studied by the second author together with M.~Baker, R.~Morrison, N.~Pflueger, and Q.~Ren \cite{bitangent1},
H.~Markwig \cite{LM17}, and H.~Lee \cite{bitangent3}.

The problem of describing the projective dual variety $X^*$ is one of implicitization, hence, the problem of describing the tropicalization $\trop(X^*)$ is one of \emph{tropical implicitization}. This has been studied by J.~Tevelev and B.~Sturmfels in the special case when the map being implicitized has an algebraic torus as its domain \cite[\S 5]{elimination}.

\section{Tropical Tangents}\label{sec:tangent}
\subsection{Tropical basics}\label{sec:basics}
Throughout this paper, we will consider a fixed hypersurface $X=V(f)\in\PP^n$ of degree $d$ for $f=\sum_{u\in\A} c_u x^u$. Here $\A$ is a subset of $\ZZ^{n+1}$, and for $u\in \A$, $c_u$ is a non-zero element of $\KK$. The Newton polytope of $f$ is 
\[
\Delta_X=\conv \A.
\]
We will frequently work with the affine cone $\widehat X$ of $X$ in $\KK^{n+1}$.

We have a tropicalization map
\begin{align*}
	\trop:(\KK^*)^{n+1}&\to \RR^{n+1}\\
(x_0,\ldots,x_n)&\mapsto(-\nu(x_0),\ldots,-\nu(x_n))
\end{align*}
and the tropicalization $\trop(\widehat{X})$ of $\widehat{X}$  is the Euclidean closure of the image of $\widehat X\cap (\KK^*)^{n+1}$ under the map $\trop$. 
The tropicalization $\trop(X)$ is the image of $\trop(\widehat{X})$ in 
\[\RR^{n+1}/\RR\cdot (1,\ldots,1)\cong\RR^n\]
where we identify $\RR^{n+1}/\RR\cdot (1,\ldots,1)$ with $\RR^n$ by projecting $\RR^{n+1}$ onto its $1$st through $n$th coordinates.
In $\RR^n$, we take $e_1,\ldots,e_n$ to be the standard basis vectors, and $e_0=\sum_{i=1}^n -e_i$. These are the images of the standard basis vectors of $\RR^{n+1}$ in our chosen identification of $\RR^{n+1}$ with $\RR^n$.

The tropical varieties $\trop(\widehat{X})$ and $\trop({X})$ are the support of pure codimension-one polyhedral complexes in $\RR^{n+1}$ and $\RR^n$, respectively. These complexes can be described quite explicitly using Kapranov's theorem, see e.g.~\cite[Theorem 3.1.3]{tropical}.\footnote{Note that we use a different sign convention than loc.~cit.}
For $p\in \RR^{n+1}$, let $\A(p)$ denote the subset of those $u\in \A$ for which 
\[ \nu(c_u)-\langle p,u \rangle\]
is minimal. 
Then 
\[
	\trop(\widehat X)=\{p\in \RR^{n+1}\ |\ \#\A(p)>1\}
\]
and the sets $\conv(\A(p))$ form a polyhedral subdivision of $\Delta_X$ which is dual to a polyhedral complex whose support is $\trop(\widehat X)$. Points $p$ and $q$ are in the same cell of this complex if $\A(p)=\A(q)$. We say that $\trop(X)$ is \emph{smooth} if the induced subdivision of $\Delta_X$ is a unimodular triangulation. 

\begin{ex}
	Consider the curve $X_2$ from Example \ref{ex:dualCurveExample}. We picture the subdivision of $\Delta_{X_2}$ by the sets $\A(p)$ in Figure \ref{fig:np}. For example, for $p=(0,3/2,1)$, we have $\A(p)=\{(1,1,1),(0,1,2)\}$. The subset of those $q\in \trop(\widehat{X_2})$ for which $\A(p)=\A(q)$ consists of the interior of
	\[
		\conv \{(0,1,1),(0,2,1)\}+\RR\cdot (1,1,1).
	\]
\end{ex}

\begin{figure}
	\begin{tikzpicture}[scale=1.6]
\draw (1,0)--(2,1)--(1,2)--(0,2)--(1,0);
\draw (1,0) -- (1,2);
\draw (0,2) -- (1,1) -- (2,1);
\node [below] at (1,0) {$x_0^2x_1$};
\node [right] at (2,1) {$t^2x_1^2x_2$};
\node [above right] at (1,1) {$x_0x_1x_2$};
\node [above] at (1,2) {$tx_1x_2^2$};
\node [left] at (0,2) {$x_0x_2^2$};
\end{tikzpicture}
\caption{The induced triangulation of $\Delta_{X_1}$}\label{fig:np}
\end{figure}
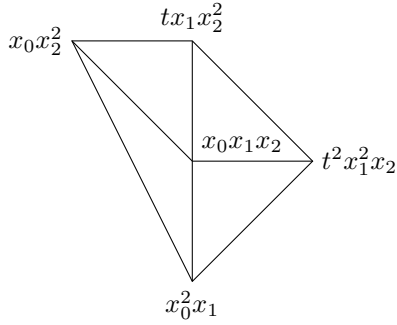

\begin{defn}\label{defn:gen}
We say that $\trop(X)$ has \emph{generic valuations} if for all $k\geq 1$, the set of all 
$p\in \trop(\widehat X)$ satisfying
\[
	\#\{\nu(c_u)-\langle p,u\rangle \ |\ u\in \A\}\leq \#\A-k
\]
has dimension at most $n+1-k$.
\end{defn}
\begin{rem}For any regular unimodular triangulation of $\Delta_X$, it is possible to find to find a hypersurface $Y$ with $\Delta_Y=\Delta_X$ which induces the same triangulation, and such that $\trop(Y)$ has generic valuations.  
\end{rem}

\subsection{Equations for tangents}
We are interested in determining the tropical tangents of $\trop(X)$:
\begin{defn}\label{def:tangent}
	A point $q\in \RR^n$ is a \emph{tropical tangent} at $p$ if there exists a smooth point $x\in X\cap (\KK^*)^n$ such that $\trop(x)=p$,  and the tangent hyperplane to $X$ at $x$ (viewed as a point of $(\PP^n)^*$) tropicalizes to $q$.
\end{defn}

For a smooth point $x\in X$, the tangent hyperplane has homogeneous coordinates $(f_0:f_1:\cdots:f_n)$, where $f_i$ denotes the partial derivative of $f$ with respect to $i$. Hence, a point $(y_0:\cdots:y_n)\in (\PP^n)^{*}$ is tangent to $X$ at a smooth point $x=(x_0:\cdots:x_n)\in \PP^n$ if and only if the system of equations
\begin{align*}
f(x)&=0\\
f_0(x)&=\lambda y_0\\
&\vdots\\
f_n(x)&=\lambda y_n
\end{align*}
is satisfied for some $\lambda\neq 0$. The point $x$ being smooth simply means that some $f_i$ is non-vanishing. In particular, if we restrict our attention to points $(y_0:\cdots:y_n)\in (\PP^n)^{*}\cap (\KK^*)^n$, a point $x$ satisfying the above equations is automatically a smooth point of $X$.

We now fix a point $p\in\trop(\widehat X)$, and will restrict our attention to those $x\in \widehat X$ with $\trop(x)=p$. 
Introducing new variables $z_i$, our previous system is equivalent to solving
\begin{align}
\begin{split}\label{eqn:main}
f(x)&=0\\
x_0f_0(x)&=z_0\\
&\vdots\\
x_nf_n(x)&=z_n
\end{split}
\end{align}
with the additional equations $z_i=x_iy_i\lambda$. 
We can rewrite \eqref{eqn:main} as 
\begin{equation}\label{matrixForma}
	\left(\begin{array}{c c c c c}
		1 &\cdots& 1& \cdots &1\\
		& \cdots & u_0 & \cdots &\\
		& & \vdots &  &\\
		& \cdots & u_n & \cdots &
	\end{array}\right)
	\left(\begin{array}{c}
		\vdots\\
		a_u\\
		\vdots
	\end{array}\right)+\HOT
	=\left(
	\begin{array}{c}
		0\\
		z_0\\
		\vdots\\
		z_n
	\end{array}\right)
\end{equation}
where the columns on the left are indexed by $u\in \A(p)$
and $\HOT$ denotes terms of valuation higher than
\[
\nu_0(p):=\nu(c_ux^u)=\nu(c_u)-\langle p,u\rangle
\]
for any $u\in \A(p)$.

Given a solution to \eqref{eqn:main} with $x=(x_0,\ldots,x_n)$ tropicalizing to $p$, set
\begin{equation}\label{eqn:s}
s_i=\nu(z_i)-\nu_0(p)
\end{equation}
for $i=0,\ldots,n$.
Considering \eqref{matrixForma}, we observe that $s_i\geq 0$.
By choosing $\lambda=a_u$ for any $u\in\A(p)$, we obtain that the solution $y$ tropicalizes to 
\[
\trop(y)=\trop(z_i)-\trop(\lambda)-\trop(x_i)=-p-s.
\]
Our question has now become: for what $s_i\geq 0$ is there a solution of \eqref{eqn:main} satisfying \eqref{eqn:s}?
It will be convenient to use the convention that $s_{-1}=\infty$.

\begin{rem}
When discussing tropical tangents of $\trop(X)$, it is convenient to fix a point in $\trop(X)$. On the other hand, when solving the above system of equations, it is more convenient to fix a point in $\trop(\widehat{X})$ which projects to the former point. We will occasionally abuse notation and use $p$ to denote both such points. In other words, we are fixing an arbitrary section of the map $\trop(\widehat{X})\to \trop(X)$.
\end{rem}

\section{Consistency}\label{sec:consistency}
To determine which $s_i$ are possible for solutions of \eqref{eqn:main}, we will transform this system of equations by inductively performing row operations on the matrix in the left hand side of \eqref{matrixForma}. 
\subsection{Motivating example}

Before showing how we will transform \eqref{eqn:main} in general, we present an example:
\begin{ex}\label{ex:mot}
Consider the cubic
	\[
		f=x_1x_2x_3+x_0x_1x_2+t(x_0x_1^2+x_0x_2^2+x_0^2x_1)
		\]
		where $t\in\KK$ is any element with $\nu(t)=1$, and set $\widehat X=V(f)\subset \KK^4$. Consider the image in $\trop(X)$ of those points $x\in \widehat X$ for which 
	\begin{align*}
\nu(x_1x_2x_3)=\nu(x_0x_1x_2)< \nu(tx_0x_1^2)= \nu(tx_0x_2^2)< \nu(tx_0^2x_1).
	\end{align*}
This is the set
	\[
		P=\{p=(p_0,p_1,p_2,p_3)\in\RR^4\ |\ p_3=p_0,p_1=p_2>p_0\}.
		\]
	We would like to determine the tropical tangents to any point in the image of $P$ in $\trop(X)$. We begin with the system \eqref{eqn:main}:
	\begin{align*}
		\begin{array}{l l  l@{}  l l l@{} l l l@{} l l l@{} l l l@{} l l}
			f	&	=	&&	x_1x_2x_3	&+&&	x_0x_1x_2	&+&&	tx_0x_1^2	&+&&	tx_0x_2^2	&+&&	tx_0^2x_1	&=0\\
			x_0f_0	&	=	&&			&&&	x_0x_1x_2	&+&&	tx_0x_1^2	&+&&	tx_0x_2^2	&+&2&	tx_0^2x_1	&=z_0\\
			x_1f_1	&	=	&&	x_1x_2x_3	&+&&	x_0x_1x_2	&+&2&	tx_0x_1^2	&&&			&+&&	tx_0^2x_1	&=z_1\\
			x_2f_2	&	=	&&	x_1x_2x_3	&+&&	x_0x_1x_2	&&&			&+&2&	tx_0x_2^2	&&&			&=z_2\\
			x_3f_3	&	=	&&	x_1x_2x_3	&&&			&&&			&&&			&&&			&=z_3\\
		\end{array}.
	\end{align*}
Subtracting the first row from the third and fourth, we obtain the equivalent system
	\begin{align*}
		\begin{array}{l l  l@{}  l l l@{} l l l@{} l l l@{} l l l@{} l l}
			f	&		=	&&	x_1x_2x_3	&+&&	x_0x_1x_2	&+&&	tx_0x_1^2	&+&&	tx_0x_2^2	&+&&	tx_0^2x_1	&=0\\
			x_0f_0	&	=		&&			&&&	x_0x_1x_2	&+&&	tx_0x_1^2	&+&&	tx_0x_2^2	&+&2&	tx_0^2x_1	&=z_0\\
			g_{1}^{(1)}	&	:=	&&			&&&			&&&	tx_0x_1^2	&+&-&	tx_0x_2^2	&&&			&=z_1\\
			g_{2}^{(1)}	&	:=	&&			&&&			&&-&	tx_0x_1^2	&+&&	tx_0x_2^2	&+&-&	tx_0^2x_1	&=z_2\\
			x_3f_3		&	=	&&	x_1x_2x_3	&&&			&&&			&&&			&&&			&=z_3\\
		\end{array}.
	\end{align*}
We modify this system once more, adding the third row to the fourth:
	\begin{equation*}
		\begin{split}\label{eqn:mot}	\begin{array}{l l  l@{}  l l l@{} l l l@{} l l l@{} l l l@{} l l}
			f	&		=	&&	x_1x_2x_3	&+&&	x_0x_1x_2	&+&&	tx_0x_1^2	&+&&	tx_0x_2^2	&+&&	tx_0^2x_1	&=0\\
			x_0f_0	&	=		&&			&&&	x_0x_1x_2	&+&&	tx_0x_1^2	&+&&	tx_0x_2^2	&+&2&	tx_0^2x_1	&=z_0\\
			g_{1}^{(1)}	&	=	&&			&&&			&&&	tx_0x_1^2	&+&-&	tx_0x_2^2	&&&			&=z_1\\
			g_{2}^{(2)}	&	:=	&&			&&&			&&&			&&&			&&-&	tx_0^2x_1	&=z_1+z_2\\
			x_3f_3		&	=	&&	x_1x_2x_3	&&&			&&&			&&&			&&&			&=z_3\\
		\end{array}.
	\end{split}\end{equation*}
By looking at the lowest order terms in this system of equations, we obtain that
\begin{align*}
\nu(z_0)=\nu(z_3)=\nu_0(p)\\
	\nu(z_1)\geq \nu(tx_0x_1^2):=\nu_1(p)\\
	\nu(z_1+z_2)= \nu(tx_0^2x_1):=\nu_2(p)
\end{align*}
which in turn implies
\begin{align*}
	&	\nu_1(p)\leq  \nu(z_1)=\nu(z_2) \leq \nu_2(p);& \textrm{or}\\
	&\nu_2(p)  \leq  \nu(z_1),\nu(z_2)\ \textrm{and}\ \nu_2(p)=\nu(z_i) \ \textrm{for either}\ i=1,2.
\end{align*}
	From our discussion of lifting in \S \ref{sec:lifting}, it will follow that these necessary conditions for tropical tangency are also sufficient after replacing $\leq$ by $<$, see Example \ref{ex:mot2}.
\end{ex}

\subsection{Consistent sequences}
We now describe how to transform our system \eqref{eqn:main} in general.
The following definition is key:
\begin{defn}
Consider a matrix $M$.
 A set $J$ indexing a subset of the rows of $M$ is $M$-\emph{consistent} if
\begin{enumerate}
\item the rowspan of $M_J$ contains no standard basis vector;
\item for $i\notin J$, the row vector $M_i$ is not in the rowspan of $M_J$.
\end{enumerate}
\end{defn}

We next set up notation:

\begin{align*}
g_i^{(0)}&=x_if_i\qquad &i=0,\ldots,n;\qquad g_{-1}^{(0)}=f\\
z_i^{(0)}&=z_i\qquad &i=0,\ldots,n;\qquad z_{-1}^{(0)}=0\\
u_i^{(0)}&=u_i\qquad &i=0,\ldots,n,\ u\in \A;\qquad u_{-1}^{(0)}=1\\
\A_{0}&=\A(p)\\
\nu_0&=\nu(a_u) \ \textrm{for any}\ u\in \A_0\\
I_{0}&=\{-1,\ldots,n\}
\end{align*}
We then let $M^{(0)}$ be the matrix with rows indexed by $I_0$ and columns indexed by $\A_0$, whose $iu$th entry is $u_i^{(0)}=u_i$.
Equation \eqref{eqn:main} now takes the form
\[
	g_i^{(0)}=z_i^{(0)}\qquad i=-1,\ldots,n
\]
and
\eqref{matrixForma} can be written as
\[
	M^{(0)}\cdot \left(\begin{array}{c} \vdots\\ c_ux^u\\\vdots\end{array}\right)+\HOT=\left(\begin{array}{c}\vdots\\ z_i^{(0)}\\\vdots\end{array}\right).
\]

We now proceed inductively on $m$.
Suppose that $J_{m}$ is an $M^{(m)}$-consistent set. In the special case $m=0$, we require that $-1\in J_{0}$. Let $K_{m}$ be the lexicographic first subset of $J_{m}$ for which 
\[\rank M_{K_{m}}^{(m)}=\rank M_{J_{m}}^{(m)}=\#K_{m}.\] 
We set
\[
I_{m+1}=J_m\setminus K_m.
\]
For each element $i\in I_{m+1}$, there are unique rational numbers $\alpha_{ij}^{(m)},j\in K_m$ such that
\[
	u_i^{(m)}=\sum_{j\in K_m} \alpha_{ij}^{(m)} u_j^{(m)}
\]
for all $u\in\A_m(p)$.
As long as $I_{m+1}\neq \emptyset$,
set
\[
	\A_{m+1}=\{u\in \A\ |\ 	u_i^{(m)}-\sum_{j\in K_m} \alpha_{ij}^{(m)} u_j^{(m)}\neq 0\ \textrm{for some}\ i\in I_{m+1}, \ \textrm{and}\ \nu(c_ux^u)\ \textrm{minimal}\}.
\]
Using this, we define
\begin{align*}
	\nu_{m+1}&=\nu(c_ux^u) \ \textrm{for any}\ u\in \A_{m+1}\\
g_i^{(m+1)}&=g_i^{(m)}-\sum_{j\in K_m}\alpha_{ij}^{(m)} g_j^{(m)}\qquad i\in I_{m+1}\\
z_i^{(m+1)}&=z_i^{(m)}-\sum_{j\in K_m}\alpha_{ij}^{(m)} z_j^{(m)}\qquad i\in I_{m+1}\\
u_i^{(m+1)}&=u_i^{(m)}-\sum_{j\in K_m}\alpha_{ij}^{(m)} u_j^{(m)}\qquad i\in I_{m+1}\\
M^{(m+1)}&=(M_{iu}^{(m+1)}),\qquad M_{iu}^{(m+1)}=u_i^{(m+1)}\qquad&i\in I_{m+1};\qquad u\in\A_{m+1}.
\end{align*}
\begin{defn}
We call a sequence $(J_0,J_1,\ldots,J_m)$ arising in this fashion a \emph{consistent sequence}.
A consistent sequence $(J_0,\ldots,J_m)$ is \emph{maximal} if for all $0\leq i \leq m$, $J_i$ is a maximal $M^{(i)}$-consistent subset of $I_i$.
\end{defn}
\begin{ex}\label{ex:mot2}
	In Example \ref{ex:mot}, the set $J_0=\{-1,1,2\}$ is the unique maximal consistent set for 
	\[
		M^{(0)}=\left(\begin{array}{c c}
			1 & 1\\
			0 & 1\\
			1 & 1\\
			1 & 1\\
			1 & 0		
		\end{array}\right).
		\]
	We obtain $K_0=\{-1\}$, $I_1=\{1,2\}$, $\A_1=\{(1,2,0,0),(1,0,2,0)\}$,  and 
	\[		M^{(1)}=\left(\begin{array}{c c}
			1 & -1\\
			-1 & 1\\
		\end{array}\right).
	\]
	The set $J_1=\{1,2\}$ is the unique maximal consistent set for $M^{(1)}$, leading to 
	$K_1=\{1\}$, $I_2=\{2\}$, $\A_2=\{(2,1,0,0)\}$, $z_2^{(2)}=z_1+z_2$, and 
	\[		M^{(2)}=\left(\begin{array}{c}
			 -1\\
		\end{array}\right).
	\]
	The unique maximal consistent set for $M^{(2)}$ is $J_2=\emptyset$, so the sequence $(J_0,J_1,J_2)$ is a maximal consistent sequence.
\end{ex}

For a consistent sequence $(J_0,\ldots,J_{m})$, we will consider the new system of equations
\begin{align*}
\begin{split}
	g_i^{(k)}&=z_i^{(k)}\qquad k\leq m\ \textrm{and}\ i\in (I_k\setminus I_{k+1})\\
	g_i^{(m+1)}&=z_i^{(m+1)}\qquad i\in I_{m+1}.\\
\end{split}\tag{$G^{(m+1)}$}
\end{align*}
Notice that $G^{(0)}$ is our original system of equations \eqref{eqn:main}, and each system $G^{(m+1)}$ is equivalent to the previous one $G^{(m)}$.
If we omit the equations 
\begin{align*}
	g_i^{(k)}&=z_i^{(k)}\qquad k\leq m\ \textrm{and}\ i\in I_k\setminus J_k\\
	g_i^{(m+1)}&=z_i^{(m+1)}\qquad i\in I_{m+1}
\end{align*}
	the system  $G^{(m+1)}$ can be written as 
\begin{equation}\label{eqn:soematrix}
	\left(\begin{array}{ c | c | c | c }
		M_{K_0}^{(0)}& * & \cdots&\cdots \\
		 0 & M_{K_1}^{(1)}&* &\cdots \\
		 0& 0 & \ddots& \\
		 0& 0 &0 & M_{K_m}^{(m)}\\
	\end{array}\right)
	\left(\begin{array}{c}
		\vdots\\
		c_ux^u\\
		\vdots
	\end{array}\right)+\HOT
	=\left(
	\begin{array}{c}
		z_i^{(0)}\\
\hline
		z_i^{(1)}\\
\hline
\vdots		\\
\hline
z_i^{(m)}\\
	\end{array}\right)
\end{equation}
where the rows are indexed by $K_0\cup K_1\cup \ldots \cup K_m$ and $\HOT$ denotes terms in the $K_i$th rows whose valuation are strictly larger than $\nu_i$.
We set $s_i^{(k)}=\nu(z_i^{(k)})-\nu_k$.

\begin{ex}
	The system of equations that we obtain from the consistent sequence $(J_0,J_1,J_2)$ in Example \ref{ex:mot2} is exactly the final system of equations in Example \ref{ex:mot}.
\end{ex}

\begin{prop}\label{prop:consistency}
	Suppose that the system $G^{(0)}$ has a solution, and let
	$J_0$ consist of those $j\in I_0$ such that $s_j^{(0)}\neq 0$. Then $J_0$ is consistent for $M^{(0)}$. More generally, consider a consistent sequence $(J_0,\ldots,J_{m-1})$ and suppose that the system $G^{(m)}$ has a solution.
Let
\[
	J_m=\{j\in I_m\ |\ s_j^{(m)}\neq 0\}.
\]
Then $J_m$ is consistent.
\end{prop}
\begin{proof}
	We consider the subsystem of $G^{(m)}$ given by $g_i^{(m)}=z_i^{(m)}$ for $i\in I_m$. 
Let $e_v$ be the row vector in $\RR^{\A_m}$ with $v$-th coordinate $1$, and all other coordinates $0$. 
If $e_v$ is a combination $\sum_{j\in J_m} a_jM^{(m)}_j$ of the rows of $M_{J_m}^{(m)}$,
then the valuation of the only monomial $c_vx_v$ in 
\[
	\sum_{j \in J_m} a_jg_j^{(m)} 
\]
	is $\nu_m$. On the other hand, 
\[
	\sum_{j \in J_m} a_jz_j^{(m)} 
\]
has valuation strictly larger than $\nu_m$, contradicting
\[
	\sum_{j \in J_m} a_jg_j^{(m)} =	\sum_{j \in J_m} a_jz_j^{(m)}.
\]
Hence, the rowspan of $M_{J_m}^{(m)}$ contains no standard basis vector.

On the other hand, suppose that for some $i\in I_m$, 
\[
	M_{i}^{(m)}=\sum_{j\in J_m} a_jM_{j}^{(m)}.
\]
Then the valuation of 
\[
	\sum_{j\in J_m} a_jg_j^{(m)}
\]
is strictly larger than $\nu_m$, so $i\in J_m$.
\end{proof}

\begin{rem}
The above proposition gives us necessary conditions that the $s_i^{(k)}$ must satisfy.
We may obtain inequalities on the original $s_i$ as follows. Fix a consistent sequence $(J_0,\ldots,J_m)$.
Consider the linear system of equations
\[
z_i^{(k+1)}=z_i^{(k)}-\sum_{j\in K_k} \alpha_{ij}^{(k)}z_j^{(k)}
\]
for $k\leq m$ and $i\in I_{k+1}$. Let $Z$ be the corresponding linear space.
 We obtain inequalities on $s_i=s_i^{(0)}$ by intersecting $\trop(Z)$ with $s_j^{(k)}\geq 0$, and $s_j^{(k)}=0$ for $j\in I_k\setminus J_k$, and projecting. 
\end{rem}

\subsection{Consistency and geometry of $\trop(X)$}
We now interpret consistency geometrically. Assume that $\trop(X)$ is smooth.
Fix any $k$-face $P\subset \trop(X)$. Let $\cJ(P)$ be the set consisting of the maximal $J\subset \{0,\ldots,n\}$ such that 
 for every $(k+1)$-face $Q$ adjacent to $P$,
\[
\langle Q \rangle \cap \langle J \rangle = \langle P \rangle \cap \langle J \rangle.
\]

\begin{lemma}\label{lemma:geco}
Let $p\in\trop(\widehat{X})$ map to the relative interior of $P$. Then for $J\in \cJ(P)$, $J^{+}=J\cup\{-1\}$ is a maximal $M^{(0)}$-consistent subset, and all maximal consistent subsets are of this form.
\end{lemma}

\begin{proof}
The rowspan of $M^{(0)}_{J^{+}}$ contains a standard basis vector $e_v$ if and only if there is $w\in \RR^{J^{+}}$ such that 
\[
\sum_{j\in J^{+}} w_ju_j=0
\]
for $u\in \A(p)\setminus\{v\}$, and $\sum_j w_jv_j=1$. Since $\trop(X)$ is smooth, this is equivalent to the existence of an adjacent $(k+1)$-face $Q$ with 
\[\langle Q \rangle=\langle \sum_{j\in J} w_je_j\rangle +\langle P \rangle.\]
Hence, $\langle Q \rangle \cap \langle J \rangle= \langle F \rangle \cap \langle J \rangle$ for all adjacent $(k+1)$-faces $F$ if and only if the rowspan of $M^{(0)}_{J^{+}}$ contains no standard basis vector.

Furthermore, since the sets $J$ we consider are chosen maximally, this guarantees that for $i\notin J^{+}$, the row vector $M_{i}^{(0)}$ is not in the rowspan of $M_{J^{+}}^{(0)}$.
\end{proof}

\section{Lifting}\label{sec:lifting}
By transforming the system of equations \eqref{eqn:main} as in \S \ref{sec:consistency}, we may obtain necessary conditions that the $s_i$ must satisfy for any tropical tangent. We wish to show that in at least some situations, any $s_i$ fulfilling these conditions does in fact arise from a tropical tangent.
\subsection{Osserman-Payne lifting}
We will make use of the following lifting result of B.~Osserman and S.~Payne. Let $Y$ and $Y'$ be subvarieties of $(\KK^*)^n$. We say that $\trop(Y)$ and $\trop(Y')$ \emph{intersect properly} at a point $q\in\trop(Y)\cap\trop(Y')$ if $\trop(Y)\cap\trop(Y')$ has codimension $\codim Y+\codim Y'$ in a neighborhood of the point $q$.
\begin{thm}[{\cite[Theorem 1.1]{lifting}}]\label{thm:lifting}
If $\trop(Y)$ and $\trop(Y')$ intersect properly at $q$, then $q\in\trop(Y\cap Y')$.
\end{thm}

We will use this result in the following incarnation:
\begin{cor}
Let $Y_1,\ldots,Y_k$ be hypersurfaces in $(\KK^*)^n$, and 
\[q\in \trop(Y_1)\cap \ldots\cap \trop(Y_k).\] Assume that in a neighborhood of $q$, the codimension of $\trop(Y_1)\cap \ldots\cap \trop(Y_k)$ is equal to $k$.
Then there exists $y\in Y_1\cap\ldots\cap Y_k$ with $\trop(y)=q$.
\end{cor}\label{cor:lifting}
\begin{proof}
Since the codimension of $\trop(Y_1)\cap \ldots\cap \trop(Y_k)$ at $q$ is $k$, it follows that for any $i$,
the codimension of $\trop(Y_1)\cap \ldots\cap \trop(Y_i)$ at $q$ is $i$.

Now assume that $q\in \trop(Y_1\cap\ldots\cap Y_i)$ for some $1\leq i < k$, and the codimension of $\trop(Y_1\cap\ldots\cap Y_i)$ at $q$ is $i$. Then in a neighborhood of $q$, $\trop(Y_1\cap\ldots\cap Y_i)$ agrees with $\trop(Y_1)\cap \ldots\cap \trop(Y_i)$, so $\trop(Y_1\cap\ldots\cap Y_i)\cap\trop(Y_{i+1})$ has codimension $i+1$ at $q$. By Theorem \ref{thm:lifting}, it follows that $q\in \trop(Y_1\cap\ldots\cap Y_{i+1})$. Furthermore, the codimension at $q$ is $i+1$, since a similar argument applies to any point $q'$ in a neighborhood of $q$. Proceeding by induction, we obtain the claim of the corollary.
\end{proof}

\subsection{Lifting tropical tangents}
As in \S \ref{sec:tangent} we have fixed a point $p\in \trop(X)$ and are interested in determining the tropical tangents to $\trop(X)$ at $p$.
We continue to use the notation of \S \ref{sec:consistency}.

Let $(J_0,\ldots,J_m)$ be a maximal consistent sequence and fix $L\subset \{0,\ldots, m\}$. 
\begin{defn}
	We say that \emph{lifting criteria hold} if the following three conditions are met:
\begin{enumerate}
	\item For $j\in L$,
		there exists $K'_j\subset I_j$ containing $K_j$ with 
\[
\rank M^{(j)}_{K_j'}=\#\A_j=\#K'_j.
\]

	\item For $j\notin L$, either $\#J_j=1$ or
		\[
\rank M^{(j)}_{J_j}=\#\A_j-1.
\]
\item Each subspace of $\RR^{n+1}$ generated by
\begin{align*}
\left\{u-v\ |\ u,v\in \A_j\ \textrm{and}\ j\notin L,\#J_j>1\right\}
\cup\left\{u\ |\ u\in \A_j\ \textrm{and}\ j\in L\right\}
\end{align*}
along with, for each $j\notin L$ with $\#J_j=1$, a vector of the form $u-v$ for $u\neq v\in \A_j$,
has dimension\label{assumpt:three} 
\[
	\sum_{j\in L} \#\A_j+\sum_{\substack{j\notin L\\ \#J_j>1}} (\#\A_j-1)+\#\{j\notin L\ |\ \#J_j=1\}.
\]
\end{enumerate}
\end{defn}

\begin{prop}\label{prop:lifting}
	Suppose that lifting criteria hold and $I_{m+1}=\emptyset$.
	For $j\notin L$, set $K'_j=K_j$.
For all $j=0,\ldots,m$ and $i\in K_j$, fix arbitrary $z_i^{(j)}$ satisfying $\nu(z_i^{(j)})>\nu_j$. Likewise, for 
	$i\in K_j'\setminus K_j$, fix $z_i^{(j)}$ satisfying $\nu(z_i^{(j)})=\nu_j$ with lowest order part generic among the $z_k^{(j)}$, $k\in K_j'$.
Then for the indices $i\in I_j\setminus (J_j\cup K_j')$ there exist $z_i^{(j)}$ satisfying $\nu(z_i^{(j)})=\nu_j$ such that the system of equations $G^{(m+1)}$ has a solution.
\end{prop}
\begin{proof}
We consider the system of equations from 
\eqref{eqn:soematrix}. Note that by our assumption on the $z_i^{(j)}$, in each row the terms with minimal valuation are only appearing on the left hand side. Putting each matrix $M_{K_j}^{(j)}$ in reduced row echelon form, for $j \notin L$ and $\#J_j>1$, the rank condition guarantees that after these row operations, the lowest order part of each equation from the $M_{K_j}^{(j)}$ block has the form \[c_ux^u-\mu c_{v{(j)}}x^{v{(j)}}\] for some trivially valued constant $\mu$ and $u\in \A_{j}$, and some fixed $v{(j)}\neq u\in \A_{j}$. 
If instead $j \notin L$ and $\#J_j=1$, the $M_{K_j}^{(j)}$ block has a single equation whose lowest order term might not be a binomial.

For $\ell\in L$, we instead include all equations $g_i^{(\ell)}$ for $i\in K_\ell'$. By performing row operations on these equations, we may obtain a new system of equations whose $K_\ell'$-equations have the form
\begin{equation}
c_ux^u+\HOT=\hat{z_i}^{(\ell)}
\end{equation}
for $u\in \A_\ell$ and $i\in K_\ell'$, where 
\begin{equation}\label{eqn:hat}
\hat{z_i}^{(\ell)}=z_i^{(\ell)}+\sum_{j\in K_{\ell}'\setminus \{i\}} \beta_{ij} z_j^{(\ell)}
\end{equation}
for some trivially valued constants $\beta_{ij}$, and $\HOT$ denotes terms of valuation larger than $\nu_\ell$.
For $i\in K_\ell$, $\hat{z_i}^{(\ell)}$ must involve some $z_j^{(\ell)}$ for $j$ not in $K_\ell$, otherwise $J_\ell$ would not be consistent. Since we have fixed sufficiently generic $z_i^{(\ell)}$ for $i\in K_{\ell}'\setminus K_\ell$, it follows that $\nu(\hat{z_i}^{(\ell)})=\nu_\ell$ for all $i\in K_\ell'$.
Thus, the lowest order terms of these equations are  $c_ux^u-\mu$  for some constant $\mu$ and $u\in \A_{\ell}$. 

Tropicalizing, each equation in this new system (indexed by $i\in K_j$ or $i\in K_\ell'$) cuts out a tropical variety contained in one of the following:
\begin{enumerate}
	\item ($j\notin L$, $\#J_j>1$) an affine hyperplane orthogonal to $u-v(j)$ with $u,v(j)\in \A_j$ 
	\item ($j\notin L$, $\#J_j=1$) a union of affine hyperplanes orthogonal to $u-v$ with $u,v\in \A_j$ 
	\item ($j\notin L$, $\#J_j>1$) an affine hyperplane orthogonal to $u$ with $u\in \A_j$ 
\end{enumerate}
Any collection of the above vectors is linearly independent by assumption \ref{assumpt:three} above, so the intersection of these tropical hypersurfaces at $q=p$ satisfies the hypotheses of Corollary \ref{cor:lifting}.
	Thus, we obtain a solution $x$ to this system of equations, and hence to the subsystem of $G^{(m+1)}$ for $i\in K_j$ or $i\in K_\ell'$. For the remaining indices $i\in I_j\setminus (J_j\cup K_j')$, we define $z_i$ to be $g_i^{(j)}$. Since $J_j$ was chosen maximally, we must have $\nu(z_i^{(j)})=\nu_j$ by Proposition \ref{prop:consistency}. Since $I_{m+1}=\emptyset$, we obtain all indices in this fashion.
\end{proof}

\begin{ex}\label{ex:mot3}
	Proposition \ref{prop:lifting} shows that the necessary conditions for tropical tangents in Example \ref{ex:mot} are also sufficient after replacing $\leq$ by $<$.
	Indeed, by Example	\ref{ex:mot2}, we are considering the maximal consistent sequence $(J_0,J_1,J_2)$ with $J_0=\{-1,1,2\}$, $J_1=\{2\}$, and $J_2=\{\emptyset\}$. To apply Proposition \ref{prop:lifting}, we set $L=\{2\}$. 
	Taking $K_2'=I_2=\{2\}$, the first and second lifting criteria are clearly met. The third is as well, since the subspace of $\RR^{n+1}$ generated by 
	\begin{align*}
		&(-1,0,0,1)=(0,1,1,1)-(1,1,1,0)\\
		&(0,2,-2,0)=(1,2,0,0)-(1,0,2,0)\\
		&(2,1,0,0)
	\end{align*}
	has dimension three as desired. Since additionally $I_3=\emptyset$, we may apply the proposition.

	After fixing $z_1^{(1)}=z_1$ with valuation larger than $\nu_1$ and $z_2^{(2)}=z_1+z_2$ with valuation $\nu_2$, by the proposition we may find $z_0,z_3$ with valuation $\nu_0$ such that $G^{(3)}$ has a solution. 

	Let $\epsilon\in \KK$ be an arbitrary element with valuation $\nu_2$. If $\nu(z_1) <\nu_2$, we may choose $z_2=\epsilon-z_1$, leading to the first kind of tangent described in Example \ref{ex:mot}. If $\nu(z_1)=\nu_2$, we may choose $z_2$ arbitrarily with valuation larger than $\nu_2$, leading to one of the second kinds of tangents. The other case is covered by taking $z_2$ of valuation $\nu_2$ and $z_1$ with valuation larger than $\nu_2$.
\end{ex}

\subsection{Tangents at vertices}\label{sec:vert}
We are now in a position to determine all tropical tangents at a vertex of $\trop(X)$ in the smooth case, and prove Theorem \ref{thm:vert}.
\begin{proof}[Proof of Theorem \ref{thm:vert}]
By Lemma \ref{lemma:geco}, the condition that $\langle E \rangle \cap \langle J \rangle =0$ is exactly the condition that $J^{+}=J\cup\{-1\}$ is $M^{(0)}$-consistent.
By Proposition \ref{prop:consistency} we know that any tropical tangent to $X$ at $p$ must be contained in the union in the statement of the theorem. 
	On the other hand, by taking the consistent sequence $(J^{+})$ with $L=\{0\}$, the lifting criteria hold and $I_1=\emptyset$. We may apply Proposition \ref{prop:lifting} to obtain the tangent $-p-s$ as long as $s_i>0$ for $i\in J$. Taking the closure, we may replace $s_i>0$ by $s_i\geq 0$.
\end{proof}

\section{Next Order Terms}\label{sec:NOT}
The situation of tropical tangents at vertices of $\trop(X)$ is especially tractable, since we may deal directly with the system $G^{(0)}$ of equations. For points $p$ in higher-dimensional strata of $\trop(X)$, we will be forced to deal with systems $G^{(m)}$ for $m>0$, which means in particular that we will need to better understand the structure of the sets $\A_j$ for $j>0$.
For our purposes, a good understanding of $\A_1$ will suffice for the most part. 

Fix a cell $P$ in $\trop(X)\subset \RR^{n}$. 
Take $p\in \RR^{n+1}$ to be any point in the relative interior of the cell $\widehat{P}$ of $\trop(\widehat{X})$ corresponding to $P$.

	For $v\in \A$ and $q\in\RR^{n+1}$, let $\phi_q(v)=\nu(c_v)-\langle q,v\rangle$. 
We set $\phi(v)=\phi_p(v)$, and 
\[\A_1'=\{u\in \A\ |\ u\ \textrm{not in affine span of}\ \A(p),\ \textrm{and}\ \phi(u)\ \textrm{minimal}\}.\]

\begin{lemma}\label{lemma:a1}
	Suppose that either $\trop(X)$ is smooth, or $n\leq 3$. Then for any $u\in \A_1'$, $\{u\}\cup \A_0$ lie in a common cell in the subdivision of $\Delta_X$.
\end{lemma}
\begin{proof}
	Let $b$ be the barycenter of $\conv \A_0$. Since $u$ is not in the affine span of $\A_0$, there is some cell $C$ in the subdivision of $\Delta_X$ containing $\A_0$ as a proper face which intersects the segment from $u$ to $b$ in its interior.
	Let $L$ be the vector space generated by $v-w$ for $v,w\in \A_0$.
	By projecting onto $\RR^{n+1}/L$, we see that there are $v_1,\ldots,v_k\in \A$ not contained in the affine span of $\A_0$, and $\alpha_i\in \QQ_{>0}$ such that 
	\[
		u=\sum_{i=1}^k \alpha_i(v_i-b)+b+w
		\]
for some $w\in L$.
	In fact, if $\trop(X)$ is smooth, $C$ is a simplex and any element of $\A$ is in the integral affine span of $C\cap \A$. It follows that we can take the $\alpha_i$ to be integral. Likewise, if $n \leq 3$, then one may similarly assume that the $\alpha_i$ are integral. This follows from the fact that the Hilbert basis of a one or two-dimensional rational cone is contained in the convex hull of the origin and its primitive generators.

	The subdivision of $\Delta_X$ is induced by the values $\nu(c_v)$ at $v\in \A$; this is the same as the subdivision induced by $\phi(v)$. Notice that $\phi(v)=\nu_0$ for any $v\in \A_0$, and $\phi(v)-\nu_0\geq 0$ for any $v\in \A$.

Suppose that $u$ is not in $C$.	For $\epsilon$ sufficiently small, $\epsilon u+(1-\epsilon)b$ is in $C$. For such $\epsilon$, convexity implies that
	\[
		\epsilon\phi(u)+(1-\epsilon)\nu_0>\epsilon \sum_i \alpha_i(\phi(v_i)-\nu_0)+\nu_0.
		\]
It follows that 
	\[\phi(u)-\nu_0>\sum_i \alpha_i(\phi(v_i)-\nu_0).\]
Since by the above we have seen that $\alpha_i\geq 1$ for some $i$, we obtain
	\[
		\phi(u)>\phi(v_i)
		\]
which contradicts $u\in \A_1'$. Thus, we must actually have $u\in C$, as desired.
\end{proof}

\begin{lemma}\label{lemma:prime}
Suppose that $\#I_1\geq \dim P$.
	Then $\A_1'=\A_1$.
\end{lemma}

\begin{proof}
	We begin by noting that no element $u\in \A$ in the affine space of $\A_0$ can be in $\A_1$. Indeed, such $u$
	is in the affine span of $\A_0$, which implies that $u_i^{(1)}=0$ for $i\in I_1$.

	The affine span of $\A_0$ has dimension $n-\dim P$. On the other hand, the affine space of all vectors $v\in \ZZ^{n+1}$ such that
$\sum v_i=0$ and 
	\[v_i-\sum_{j\in K_m}\alpha_{ij}^{(0)}u_j=0\]
	for $i\in I_1$ has dimension $n-\#I_1$, and it contains the affine span of $\A_0$. By comparing dimensions, it follows that these two affine spaces are equal.
	Incidentally, this shows that we must always have $\#I_1\leq \dim P$.
	
	We conclude that if $u$ is not in the affine span of $\A_0$,  $u_i^{(1)}\neq 0$ for some $i\in I_{m+1}$. The claim follows.
\end{proof}

We will use the above two lemmas to better understand how $\A_1$ and $\nu_1$ vary as $p$ moves about in the cell $P$. In the following, we will assume that $\trop(X)$ is smooth. Then facets of $P$ are in bijection with $u\in \A\setminus \A_0$ such that $\{u\}\cup \A_0$ is a simplex in the triangulation of $\Delta_X$. For such $u$, the corresponding facet $Q$ is cut out by intersecting the affine span of $P$ with the image in $\RR^{n+1}/\RR$ of  
\[
	\{q\in \RR^{n+1}\ |\ \nu(c_u)-\langle q, u\rangle = \nu(c_v)-\langle q,v\rangle \}
	\]
where $v$ is any element of $\A_0$.

\begin{lemma}\label{lemma:distance}
Consider $u$ as above with associated facet $Q$ of $P$. 
	Let $\dd_Q$ be the lattice distance to $Q$ as defined in \S \ref{sec:surfacesintro}. For any $v\in \A_0$ and $q$ in the affine span of $P$, 
	\[\dd_Q(q)=\phi_q(u)-\phi_q(v).\]
\end{lemma}
\begin{proof}
The primitive inward normal vector $v_Q$ to $Q$ lives in the quotient of $\ZZ^{n+1}$ by the sublattice $L$ generated by $v-w$, $v,w\in \A_0$.  
This is proportional to the projection of $v-u$ to $\ZZ^{n+1}/L$ for any $v\in \A_0$. Furthermore, since $\trop(X)$ is smooth, the projection of $v-u$ is in fact primitive, so this is exactly $v_Q$.

	Let $q'$ be any point in $Q$. Since $q'\in Q$, $\phi_{q'}(u)=\phi_{q'}(v)$.
We obtain
	\begin{align*}
		\dd_Q(q)&=
		\langle q-q',v_Q\rangle=
		\langle -q,-v_Q\rangle
		-\langle -q',-v_Q\rangle\\
		&=(\phi_{q}(u)-\phi_{q}(v))-(\phi_{q'}(u)-\phi_{q'}(v))\\
		&=\phi_{q}(u)-\phi_{q}(v).
	\end{align*}
\end{proof}

\begin{prop}\label{prop:a1}
	Assume that $\trop(X)$ is smooth, and $\#I_1\geq \dim P$. Let $p$ be in the relative interior of $P$. Then $\A_1$ consists of those $u\in\A$ corresponding to facets $Q$ of $P$ whose lattice distance is closest to $p$. Furthermore, $\nu_1-\nu_0$ is equal to the lattice distance of $p$ to any such facet.
\end{prop}
\begin{proof}
	By Lemmas \ref{lemma:a1} and \ref{lemma:prime}, any element $u$ of $\A_1$ is such that $\{u\}\cup \A_0$ is contained in a simplex in the triangulation of $\Delta_X$, and thus corresponds to some facet of $P$. 
	Conversely, any facet of $P$ corresponds to an element $u$ of $\A$ not contained in the affine span of $\A_0$.

	For any point $x$ tropicalizing to $p$, the valuation $\nu(c_ux^u)$ is exactly $\phi(u)$. The claims now follow by Lemma \ref{lemma:distance}.
\end{proof}

\begin{ex}
Consider 
\[
f=tx_0^3+x_0^2x_1+x_0^2x_2+x_0x_1x_2+t^3x_1^2x_2+t^2x_2^2x_1+x_1x_2x_3.
\]
Then $X=V(f)$ is a cubic surface, and 
\[
	P=\conv \{0,3e_1,2e_2,3e_1+2e_2\}
\]
is a two-face in $\trop(X)$ with $\A_0=\{(1,1,1,0),(0,1,1,1)\}$.
The edges of $P$ correspond to the vectors $(2,1,0,0)$, $(2,0,1,0)$, $(0,2,1,0)$, and $(0,1,2,0)$ as pictured in Figure \ref{fig:a1}. We have drawn the subdivision of $P$ showing which edges have closest lattice distance. For example, at the point $p=e_1+e_2$, we obtain that 
\[
	\A_1(p)=\{(2,1,0,0),(2,0,1,0),(0,1,2,0)\},
\]
and $\nu_1-\nu_0=1$. Indeed, at this $p$,
\begin{align*}
\nu(tx_0^3)&=1&\qquad\nu(x_0^2x_1)&=-1\\
\nu(x_0^2x_2)&=-1&\qquad \nu(x_0x_1x_2)&=-2\\
\nu(t^3x_1^2x_2)&=0&\qquad \nu(t^2x_2^2x_1)&=-1\\
\nu(x_1x_2x_3)&=-2
\end{align*}
so $\nu_0=-2$ and $\nu_1=-1$.
Notice that in this example, for any $p$ in the interior of $P$, $(3,0,0,0)$ will never be in $\A_1(p)$.
\end{ex}
\begin{figure}
	\begin{tikzpicture}
\draw[fill,color1] (0,0) -- (3,0) -- (3,2) -- (0,2) -- (0,0);
\draw (0,0) -- (1,1) -- (0,2);
\draw (3,0) -- (2,1) -- (3,2);
\draw (1,1) -- (2,1);
\node [below ] at (1.5,0) {\scriptsize$(2,1,0,0)$};
\node [above ] at (1.5,2) {\scriptsize$(0,1,2,0)$};
\node [left ] at (0,1) {\scriptsize$(2,0,1,0)$};
\node [right ] at (3,1) {\scriptsize$(0,2,1,0)$};
\node [below right] at (1,1) {$p$};
\draw[fill] (1,1) circle [radius=0.05];
\end{tikzpicture}
	\caption{Edges of $P$ correspond to possible elements of $\A_1$}\label{fig:a1}
\end{figure}
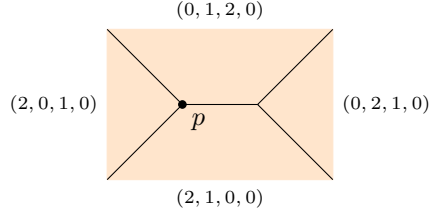

\section{Edges}\label{sec:edges}
We are now in a position to determine all tropical tangents at an edge of $\trop(X)$ in the smooth case, and prove Theorem \ref{thm:edge}.
\begin{proof}[Proof of Theorem \ref{thm:edge}]
Consider any tropical tangent 
\[
-p-\sum_i s_ie_i	
\]
to $p$. This determines a consistent sequence $(J_0,J_1,\ldots)$ by \ref{prop:consistency}. Since $E$ is an edge, $\#(J_0\setminus K_0)\leq 1$, which implies that $I_2=\emptyset$, so we only need to consider the sequence $(J_0,J_1)$. 
Since $E$ is an edge, either the rows of $M_{J_0}^{(0)}$ are linearly independent, or have a unique linear dependency:
\[
	0=\sum_{i\in Z} \alpha_iM_i^{(0)}
\]
with all $\alpha_i\neq 0$.

If the rows are linearly independent, then $K_0=J_0$ and $J_1=\emptyset$, and the only constraints on $s_i$ we obtain are exactly $s_i>0$ for $i\in J_0$.
If there is a linear dependency as above, then $K_0$ consists of every element of $J_0$ except for the final element $\ell$ of $Z$. Then $J_1$ is either empty, or consists of $\ell$. The latter is only possible if $\#\A_1>1$. In either case, we have
\[
z_\ell^{(1)}=z_\ell-\sum_{\substack{i\in Z\\i\neq \ell}} \frac{\alpha_i}{\alpha_\ell}z_i.
\]
The constraint $\nu(z_\ell^{(1)})= \nu_1$ in the case $J_1=\emptyset$  leads to the constraint that the minimum of \[\{\nu(z_i)\}_{i\in Z}\cup\{\nu_1\}\]
is obtained at least twice.	
Similarly, the constraint $\nu(z_\ell^{(1)})>\nu_1$ in the case $J_1\neq \emptyset$ leads to the constraint that the minimum of 
 \[\{\nu(z_i)\}_{i\in Z}\]
is obtained at least twice, or the minimum is at least $\nu_1$. We thus obtain the following constraints on the $s_i$.
Firstly, 
\begin{align*}
&\textrm{$s_i\geq 0$, with $s_i>0$ if and only if $i\in J_0$.}
\end{align*}
Secondly,
\begin{align}\label{eqn:ss}\begin{split}
&\textrm{$s_i=s_j\leq \nu_1-\nu_0$ for some $i\neq j\in Z$ and $s_i\leq s_k$ for all $k\in Z$; or}\\
\\
&\textrm{$s_i\geq \nu_1-\nu_0$ for all $i \in Z$, and}\\
&\textrm{$s_i=\nu_1-\nu_0$ for at least one $i\in Z$ unless $J_1\neq \emptyset$.}
\end{split}
\end{align}
If $(J_0,J_1)$ is not maximal, we may find a maximal consistent sequence $(\widetilde J_0,\widetilde J_1)$ such that $J_0\subset \widetilde J_0$, and $J_0\cup\
J_1\subset \widetilde J_0\cup \widetilde J_1$. 
For this new sequence $(\widetilde J_0,\widetilde J_1)$, the new constraints we would obtain on the $s_i$ (after relaxing strict inequalities to weak ones) are weaker than the constraints coming from $(J_0,J_1)$. We thus obtain that any tropical tangent to $p$ satisfies the constraint \eqref{eqn:ss}, and $s_i=0$ if $i\notin J_0$, for some maximal consistent sequence $(J_0,J_1)$. For $(J_0,J_1)$ maximal,  $J_1$ is already determined by $J_0$.

		By Lemma \ref{lemma:geco}, the condition that $\langle F \rangle \cap \langle J \rangle =\langle E \rangle \cap \langle J \rangle$ for all adjacent two-faces $F$ is exactly the condition that $J^{+}=J\cup\{-1\}$ is $M^{(0)}$-consistent.
Given such a maximal $J$, we set $J_0=J^{+}$. Recall that $J'$ is the minimal subset of $J$ such that $\langle E \rangle \subset \langle J'\rangle$ if such a subset exists, and $\emptyset$ otherwise. It is straightforward to check that $J'=Z\setminus\{-1\}$.

By Proposition \ref{prop:a1}, $\nu_1-\nu_0$ is the lattice distance $\dd(p)$ of $p$ to the nearest endpoint of $E$. Furthermore, $\#\A_1=2$ if and only if $p$ is the midpoint of $E$. Note that if $E$ has no endpoints, then it must be the case that $\A_1=\emptyset$, and we may view $\nu_1-\nu_0$ as being $\infty$.
The above discussion thus implies that any tropical tangent of $p$ is contained in the set described in the statement of the theorem.

It remains to show that every such point is in the closure of the set of tropical tangents. To that end, fix $J\in \cJ(E)$, and $s_j>0$ for $j\in J$ satisfying the hypotheses of the theorem. We take $J_0=J^{+}=J\cup\{-1\}$.
If $J'=\emptyset$, then it follows that $K_0=J_0$, so $I_1=\emptyset$.
In this case, the consistent sequence $(J_0)$ with $L=\{0\}$ satisfies the lifting criteria, and we may apply Proposition \ref{prop:lifting} to obtain the tangent $-p-\sum_{j\in J}s_je_j$. 
Suppose instead that $J'\neq \emptyset$, and let $J_1$ be the unique maximal consistent set for $M^{(1)}$. Note that if $\A_1=\emptyset$, we have $J_1=\emptyset$. 
Similar to above, the consistent sequence $(J_0,J_1)$ will satisfy the lifting criteria with $L=\{0,1\}$ if $\#\A_1=1$ or with $L=\{0\}$ if $\#\A_1=0,2$. Fix $z_i$ for $i\in J$ with $\nu(z_i)=\nu_0+s_i$. Given the constraints on $s_i$, we may choose these $z_i$ such that $\nu(z_\ell^{(1)})=\nu_1$ (if $p$ is not a midpoint) or $\nu(z_\ell^{(1)})>\nu_1$ (if $p$ is a midpoint). Applying Proposition \ref{prop:lifting}, we conclude that $-p-\sum_{j\in J}s_je_j$ is a tropical tangent to $p$.

We had assumed the constraint $s_i>0$, but by passing to the closure of the set of tropical tangents, we may relax this to $s_i\geq 0$. This shows that every point described in the statement of the theorem is in the closure of the set of tropical tangents.
\end{proof}

\section{Multiplicities}\label{sec:mult}
In this section, we describe a strategy for computing the multiplicities of the maximal cells in the tropical dual variety $\trop(X^*)$ for hypersurfaces of arbitrary dimension, and make these calculations explicit in the case of curves.  We refer the reader to \cite[\S 2.3]{lifting} for details on tropical multiplicities.

As usual, $X\subseteq\PP^n$ is a hypersurface and  $X^*\subseteq(\PP^n)^*$  its dual.
We consider the \emph{conormal variety}
\[
	W_X=\overline{\{(x,H)\in\PP^n\times(\PP^n)^*\ |\ x\in X\ \textrm{smooth},\ H\ \textrm{tangent to}\ X\ \textrm{at}\ x\}}.
\]
This comes with a projection $\pi:W_X\to X^*$.

We fix a cell $\sigma$ of $\trop(X^*)$, and are interested in computing its multiplicity.
Let $\sigma'$ be any cell of $\trop(W_X)$ that maps to $\sigma$ under $\trop(\pi)$. 
For such a cell $\sigma'$, consider any subvariety $Z\subset W_X$ of complimentary dimension such that $\trop(Z)$ is an affine linear space with multiplicity one which intersects $\sigma_W$ properly. We refer to such a $Z$ as a \emph{test variety}. Denote by $\eta(\sigma')$ the number of intersection points (counted with multiplicity) of $W_X$ and $Z$ which tropicalize to $\sigma'$. Likewise, denote by $\omega(\sigma')$ the tropical intersection multiplicity of $\trop(W_X)$ and $\Lambda$, where $\Lambda$ is the tropical variety whose support is the affine span of $\sigma'$, and which has constant multiplicity one.

\begin{lemma}\label{lem:weight}
Assume that $\defect(X)=0$ and that for any $\sigma_W$, the image of the lattice points in $\sigma_W$ generate the affine lattice of lattice points in $\sigma$. Then the multiplicity $\mult(\sigma)$ of $\sigma$ is the sum of 
\[
\eta(\sigma')/\omega(\sigma')
\]
 as  $\sigma'$ varies over  the cells of $\trop(W)$ mapping to $\sigma$. 
\end{lemma}
\begin{proof}
	By projective duality, $W_X$ coincides with the conormal variety $W_{X^*}$ for $X^*$, see \cite[\S1.2.A]{tevelev}. Since $\defect(X)=0$, $X^*$ is also a hypersurface, so the map $\pi:W_X \to X^*$ is an isomorphism over the smooth locus of $X^*$. In particular, $\pi$ is a generically finite map of degree one.

	We now apply \cite[Theorem C.1]{lifting}, which states that (under the assumptions we have made on the lattice points of $\sigma_W$) the multiplicity $\mult(\sigma)$ of $\sigma$ is equal to the sum of the multiplicities $\mult(\sigma')$ of the $\sigma'$.
By \cite[Theorem 5.1.1]{lifting}, $\eta(\sigma')$ is the sum of the tropical intersection multiplicities of $\trop(W_X)$ and $\trop(Z)$ at the unique point in $\sigma_W\cap\trop(Z)$. But this tropical intersection multiplicity is just the tropical intersection multiplicity of $\omega(\sigma')$ of $\Lambda$ with $\trop(Z)$ times the multiplicity $\mult(\sigma')$ of $\sigma'$.
\end{proof}

Fix now some $\sigma'\subset \trop(W_X)$ projecting to $\sigma$. We wish to compute $\eta(\sigma')/\omega(\sigma')$.
Coordinates for the torus of $\PP^n\times (\PP^n)^*$ are given by 
\[x_0/x_i,\ldots,x_n/x_i,y_0/y_j,\ldots,y_n/y_j\]
for some choice of $i$ and $j$, omitting $x_i/x_i$ and $y_j/y_j$. Equations for \[W_X\cap (\KK^*)^{(n-1)+(n-1)}\] can be obtained from any system $G^{(m)}$ by de-homogenizing with respect to $x_j$. If we instead wish to work with coordinates $x_0,\ldots,x_n,z_0,\ldots,z_n$, we may alternatively obtain equations for 
$W_X\cap (\KK^*)^{(n-1)+(n-1)}$ from $G^{(m)}$ by setting $x_i=1$.

Assume now that the $n+3$ equations we obtained for $W_X$ have binomial lowest order terms on $\sigma'$, as do $n-1$ equations cutting out $Z$. Considering the differences of the exponent vectors for each binomial $h$, we obtain $2n+2$ elements $w_h$ of $\ZZ^{n+1}\times \ZZ^{n+1}$. 
Let $A$ be the $(2n+2)\times (2n+2)$ matrix whose columns are indexed by the $2n+2$ above binomials, and the column $h$ consists of the vector $w_h$.

\begin{prop}\label{prop:weights}
If the matrix $A$ has full rank, then the quantity $\eta(\sigma')$ equals $|\det(A)|$.
\end{prop}
\begin{proof}
	By \cite[Corollary 5.2.4]{lifting}, the intersection multiplicity $\eta(\sigma')$ equals $(2n+2)!\cdot V(Q_1,Q_2,\ldots,Q_{2n+2})$, where $V$ stands for the mixed volume and $Q_i$ is the convex hull of $0$ and the vector $w_i$ from above. But since the $Q_i$ are one-dimensional, \cite[Section 5.4]{fulton}, implies that  
	\begin{align*}
	(2n+2)!\cdot &V(Q_1,Q_2,\ldots,Q_{2n+2}) =  \vol(Q_1+Q_2+\ldots+Q_{2n+2}) = |\det (A)|.
	\end{align*}
\end{proof}

We illustrate how Proposition \ref{prop:weights} and Lemma \ref{lem:weight} apply in the case of a plane curve $X$ of degree at least two.
We will always make use of the modified system of equations used in the proof of Proposition \ref{prop:lifting}.

\begin{ex}[Non-standard edges]\label{ex:nonstd}
	Let $E$ be an edge of $\trop(X)$ which is not parallel to $e_0,e_1,e_2$. By Theorem \ref{thm:edge}, $E$ contributes exactly $\sigma=-E$ to $\trop(X^*)$. In fact, the unique cell of $\trop(W_X)$ mapping onto $-E$ in $\trop(X^*)$ is $\sigma'=\{(p,-p)\ |\ p\in E\}$. This implies in particular that the lattice hypothesis in Lemma \ref{lem:weight} is satisfied.

On this cell of $W_X$, the lowest order terms for equations of $W_X$ are given by
\begin{align*}
c_ux^u+c_vx^v&=\HOT\\
\alpha_0x^u-z_0&=\HOT\\
\alpha_1x^u-z_1&=\HOT\\
\alpha_2x^u-z_2&=\HOT\\
x_0-1&=0
\end{align*}
where $u,v$ are the elements of $\A_0$ and $\alpha_i$ are non-zero constants.
We choose the test variety $Z=V(x^az^b-x^bz^a)$ for any $a,b\in \ZZ^{n+1}$ satisfying 
\[
	\det \left(\begin{array}{c c c}
		u-v& e_0 & a-b\end{array}\right)=1.
\]
This necessarily exists, since $u-v$ is primitive.

We then obtain the matrix 
\[
A=\left(\begin{array}{c c c c c c}
		u-v& u & u & u & e_0 & a-b\\
		0 & -e_0 & -e_1 & -e_2 & 0 &  b-a\end{array}\right).
\]
Performing invertible column operations, we obtain
\[
\left(\begin{array}{c c c c c c}
		u-v& u & u & u & e_0 & a-b\\
		0 & -e_0 & -e_1 & -e_2 & 0 & 0 \end{array}\right).
\]
which is easily seen to have determinant one. Thus, $\eta(\sigma')=1$. It follows that $\omega(\sigma')=1$, so $\mult(\sigma)=1$.
\end{ex}

\begin{ex}[Vertices]\label{ex:vert}
	Let $p$ be a vertex of $\trop(X)$ with no edge adjacent to $p$ parallel to $e_0$. By Theorem \ref{thm:vert}, $p$ contributes the cell $\sigma=-p-\RR_{\geq 0}\cdot e_0$ to $\trop(X^*)$, and the cell $\sigma'=(p,-p)-\RR_{\geq 0}\cdot (0,e_0)\subset \trop(W_X)$ maps to it.

In this case, we choose 
$Z=V(x_0z_2 - x_2z_0)$, leading to the matrix 
\[
A=\left(\begin{array}{c c c c c c}
		u-w& v-w & w & w & e_0 & e_0-e_2\\
		0 & 0 & -e_1 & -e_2 & 0 & e_2-e_0 \end{array}\right).
\]
Here $u,v,w$ are the elements of $\A(p)$. Since $p$ is a smooth vertex of $\trop(X)$, 
\[
	|\det \left(\begin{array}{c c c}
		u-v& v-w & e_0\end{array}\right)|=1.
\]
so $|\det(A)|=1$ as well. Similar to Example \ref{ex:nonstd}, we obtain that the contribution of $\sigma'$ to $\mult(\sigma)$ is $1$.
\end{ex}

\begin{ex}[Standard edges]\label{ex:std}
	Let $E$ be an edge of $\trop(X)$ which is parallel to $e_0$, and  bounded in direction $-e_0$ with endpoint $q$. Let $q_\midp$ be the midpoint of $E$ if the edge is bounded, otherwise we treat it as a point at infinity.
	By Theorem \ref{thm:edge}, $E$ contributes exactly $\sigma=-E-\RR_{\geq 0}\cdot e_0$ to $\trop(X^*)$. If $E$ is bounded, there are two corresponding cells of $\trop(W_X)$:
	\begin{align*}
		\sigma'=	\{ (p,-p-(p-q)\ |\ p\ \textrm{between}\ q,q_\midp\}\\
		\sigma''=(q_\midp,-q_\midp-(q_\midp-q))-\RR_{\geq 0}\cdot (0,e_0)
	\end{align*}
If $E$ is unbounded, only the first cell $\sigma'$ exists. In both cases, the lattice hypothesis in Lemma \ref{lem:weight} are satisfied.

We will again choose 
$Z=V(x_0z_2 - x_2z_0)$.
Consider first $\sigma'$. Let $u,v$ be the elements of $\A_0$, and $w$ the unique element of $\A_1$, corresponding to the endpoint $q$ of $E$.
Our system of equations for $W_X$ is
\begin{align*}
c_ux^u+c_vx^v&=\HOT\\
\alpha_0x^w-z_0&=\HOT\\
\alpha_1x^u-z_1&=\HOT\\
\alpha_2x^u-z_2&=\HOT\\
x_0-1&=0
\end{align*}
leading to the matrix 
\[
A=\left(\begin{array}{c c c c c c}
		u-v& w & u & u & e_0 & e_0-e_2\\
		0 & -e_0 & -e_1 & -e_2 & 0 & e_2-e_0 \end{array}\right).
\]
Applying invertible column operations, we obtain
\[
\left(\begin{array}{c c c c c c}
		u-v& w & u & u & e_0 & (u-w)+(e_0-e_2)\\
		0 & -e_0 & -e_1 & -e_2 & 0 & 0 \end{array}\right).
\]
But then 
\[
|\det A|= |\det
\left(\begin{array}{c c c }
	u-v & e_0 & e_0-e_2
		\end{array}\right)
+
\det\left(\begin{array}{c c c }
	u-v & e_0 & u-w
		\end{array}\right)
|=1+1=2
\]
so $\eta(\sigma')=2$. 
On the other hand, it is straightforward to verify that
the tropical variety $\Lambda$ intersects $\sigma'$ tropically transversally, implying that $\omega(\sigma')=1$. We conclude that $\mult(\sigma')=2$.

Considering instead $\sigma''$, we obtain the matrix
\[
A=\left(\begin{array}{c c c c c c}
		u-v& w-w' & u & u & e_0 & e_0-e_2\\
		0 &  0 & -e_1 & -e_2 & 0 & e_2-e_0 \end{array}\right)
\]
where $w,w'$ are the two elements of $\A_1(q_\mid)$.
Since $w$ and $w'$ correspond to the endpoints of $E$ and $\trop(X)$ is tropically smooth, we obtain 
\[
|\det A|=|\det\left(\begin{array}{c c c }
	u-v & w-w' & e_0
		\end{array}\right)|=2.
\]
Again, $\omega(\sigma')=1$ and we conclude that $\mult(\sigma')=2$.

\end{ex}
\section{Curves}\label{sec:curves}
We are now able to prove our Theorem \ref{thm:dualcurve}, describing the tropicalization of the dual of a plane curve.

\begin{proof}[Proof of Theorem \ref{thm:dualcurve}]
	The description of the underlying set $\trop(X^*)$ follows from Theorems \ref{thm:vert} and \ref{thm:edge}. Indeed, these theorems imply that the set described in the statement of the theorem is contained in $\trop(X^*)$. Furthermore, any plane $H$ with dual coordinates in $(\KK^*)^2$ which is tangent at a point $x\in X\cap (\KK^*)^2$ must be contained in this set. The only tangents we must still deal with are those which are tangent at points in the boundary of $X$. But these will form a zero-dimensional set in $\trop(X^*)$, so since the set described in the theorem is closed in the Euclidean topology and $\trop(X^*)$ is a tropical curve, the union we describe must equal $\trop(X^*)$.

It remains to calculate the multiplicities. But after applying Lemma \ref{lem:weight}, this follows from Examples \ref{ex:nonstd}, \ref{ex:vert}, and \ref{ex:std}.
\end{proof}

We next consider the situation of the Newton polygon of $X^*$. To begin with, we assume that  $\trop(X)$ is a smooth tropical curve:

\begin{prop}\label{prop:newton}
Let $\Delta_X\subset\RR^2$ be the Newton polygon of $X$, and assume that $\trop(X)$ is a smooth tropical curve. Label the vectors of $\Delta_X$ by $v_1,v_2,\ldots,v_m$ in counterclockwise orientation, omitting all edges parallel to $w_0=(-1,1),w_1=(0,-1),w_2=(1,0)$. For $i=0,1,2$, let $\sigma_i$ be the sum of edge vectors parallel to $w_i$. Then the edge vectors of the Newton polygon $\Delta_{X^*}$ of $X^*$ are exactly 
\[
-v_m,-v_{m-1},\ldots,-v_1
\]
along with 
\[
\vol(\Delta_X)\cdot w_i-\sigma_i
\]
where $\vol(\Delta_X)$ is the normalized lattice volume of $\Delta_X$.
\end{prop}

\begin{proof}
The edges of $\Delta_{X^*}$ are determined by the unbounded edges of $\trop(X^*)$. 
Let $\cS$ be the regular unimodular triangulation of $\Delta_X$ induced by the valuations of the coefficients of $f$. The tropical curve $\trop(X)$ is dual to this triangulation.

We use the explicit description of $\trop(X^*)$ from Theorem \ref{thm:dualcurve} to see: the number of rays in direction $e_i$ (counted with multiplicity) is given by $2\cdot \alpha_i+\beta_i$ where $\alpha_i$ is the number of edges of $\trop(X)$ parallel to $e_i$ excluding rays in direction $-e_i$, and $\beta_i$ is the number of vertices of $\trop(X)$ not having an edge parallel to $e_i$.
We can reinterpret $2\cdot \alpha_i$ as the number of vertices of $\trop(X)$ having an edge parallel to $e_i$, plus the number of rays of $\trop(X)$ in direction $e_i$, less the number of rays of $\trop(X)$ in direction $-e_i$.

Let $\cS'$ consist of those triangles in $\cS$ with an edge orthogonal to $e_i$, and $\cS''=\cS\setminus \cS'$. Then $\#\cS''=\beta_i$, and $\#\cS'$ is the number of vertices of $\trop(X)$ having an edge parallel to $e_i$.
Let $\kappa_i=\sigma_i/w_i$; this is the number of rays of $\trop(X)$ in direction $-e_i$, less the number of rays of $\trop(X)$ in direction $e_i$.
Then
\[2\alpha_i=\#\cS'-\kappa_i.\]
Since $\vol(\Delta_X)=\#\cS=\#\cS'+\#\cS''$, the formula for the length of the edge of $\Delta_{X^*}$ in direction $w_i$ follows.

For $u\notin\{e_0,e_1,e_2\}$, 
the number of rays in direction $u$ (counted with multiplicity) is given by the number of rays of $\trop(X)$ in direction $-u$. The description of the other edge vectors of $\Delta_{X^*}$ now follows.
\end{proof}

\begin {ex}[Example \ref{ex:dualCurveExample} continued]
	We illustrate the proof of Proposition \ref{prop:newton} using the plane curves $X_1$ and $X_2$ from Example \ref{ex:dualCurveExample}. The Newton polygons of $X_1^*$ and $X_2^*$ agree and were pictured in Figure \ref{fig:NewtonPolygon}.
 
	In Figure \ref{fig:SubdividedNewtonPolygons}, we show the subdivided Newton polygons of $X_1$ and $X_2$. The shaded regions form the set $\cS'$ with respect to $e_1$, and the unshaded regions form the set $\cS''$. We see that for $X_1$, $\beta_1=2$, and $2\alpha_1+\kappa_1=2$, and $\kappa_1=0$. Hence, we obtain that the length of the edge in direction $w_1$ in $\Delta_{X_1^*}$ is $4$. On the other hand, for $X_2$ we have $\beta_1=4$, and $2\alpha_1+\kappa_1=0$, and $\kappa_1=0$. This also leads to an edge of length $4$; in both cases, $\beta_1+2\alpha_1+\kappa_1$ is $4$, the area of $\Delta_{X_i}$.
\end{ex}

  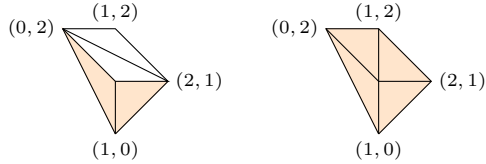
\begin{figure}[h]
\centering
\begin{tikzpicture}[scale=.7]
	\draw [fill,color1] (1,0) -- (2,1) -- (1,1) -- (0,2) -- (1,0);
\draw (1,0)--(2,1)--(1,2)--(0,2)--(1,0);
\draw (1,0)--(1,1)--(0,2);
\draw (1,1)--(2,1)--(0,2);
{\scriptsize
\node [below] at (1,0) {$(1,0)$};
\node [right] at (2,1) {$(2,1)$};
\node [above] at (1,2) {$(1,2)$};
\node [left] at (0,2) {$(0,2)$};
}

\begin{scope}[shift={(5,0)}]
\draw [fill,color1] (1,0)--(2,1)--(1,2)--(0,2)--(1,0);
\draw (1,0)--(2,1)--(1,2)--(0,2)--(1,0);
\draw (1,0) -- (1,2);
\draw (0,2) -- (1,1) -- (2,1);

{\scriptsize
\node [below] at (1,0) {$(1,0)$};
\node [right] at (2,1) {$(2,1)$};
\node [above] at (1,2) {$(1,2)$};
\node [left] at (0,2) {$(0,2)$};
}
\end{scope}
\end{tikzpicture}
\caption{The subdivided Newton polygons of $X_1$, and  $X_2$.}
\label{fig:SubdividedNewtonPolygons}
\end{figure}

To generalize Proposition \ref{prop:newton} to curves $X$ with $\trop(X)$ not smooth, we need the following lemma:

\begin{lemma}\label{lemma:constant}
	Let $X$ be an irreducible hypersurface in $\PP^n$, and let $\HH$ be the space of all hypersurfaces in $\PP^n$ with Newton polytope $\Delta_X$. Then there is a non-empty Zariski open subset $U\subset \HH$ such that for $Y\in U$, $\Delta_{Y^*}$ does not depend on $Y$.
\end{lemma}

\begin{proof}
Let 
\[
	g=\sum_{u\in \Delta_X\cap \ZZ^{n+1}} \alpha_ux^u\in \KK(\alpha_u\ |\ u\in\Delta_X\cap \ZZ^{n+1})
\]
where the $\alpha_u$ are indeterminates. The coefficients $\alpha_u$ may be specialized to $a_u\in \KK$ obtain the equation of any hypersurface $Y\in \HH$. 

Consider the ideal 
\[
	I=\left\langle \frac{\partial g}{\partial x_0}-y_0,\frac{\partial g}{\partial x_1}-y_1,\ldots,
	\frac{\partial g}{\partial x_n}-y_n\right\rangle\subset \KK(\alpha_u)[x_0,\ldots,x_n,y_0,\ldots,y_n].
\]
Then $I\cap  \KK(\alpha_u)[y_0,\ldots,y_n]$ is a principal ideal, say with irreducible generator $h$, and the projective dual of $V(g)$ (working over $\overline{\KK(\alpha_u)}$) is exactly $V(h)$. 

For any $Y\in \HH$, the ideal of $Y^*$ will contain $h$, after we specialize the $\alpha_u$ to the coefficients $a_u$ defining $Y$. In particular, the specialization $h(a_u)$ will cut out $Y^*$ whenever this specialization is irreducible and non-zero. But the locus of $a_u\in \KK$ for which $h(a_u)$ is irreducible and non-zero is Zariski open \cite{noether}. Hence, there is a Zariski open subset $U'\subset \HH$ for which, given $Y\in U$, the equation for $Y^*$ is $h(a_u)$. Within $U'$, there is a Zariski open subset $U$ for which the Newton polytope of the specializations $h(a_u)$ are the same as the Newton polytope of $h$. The claim of the lemma now follows.
\end{proof}

Corollary \ref{cor:newton} now follows:
\begin{proof}[Proof of Corollary \ref{cor:newton}]
	Consider the set $U\subset \HH$ from Lemma \ref{lemma:constant}. The set of all $Y\in \HH$ with $\trop(Y)$ smooth is dense in $\HH$, hence it must intersect $U$. Thus, for $X$ sufficiently generic with respect to $\Delta_X$, $X$ is in $U$, and $\Delta_{X^*}$ will be the same as $\Delta_{Y^*}$ for some $Y\in \HH$ with $\trop(Y)$ smooth. But we may apply Proposition \ref{prop:newton} to obtain the description of $\Delta_{Y^*}$.
\end{proof}

From the Newton polygon $\Delta_{X^*}$, we may read off the degree of $X^*$. We first illustrate this with several examples.

\begin{ex}[Generic plane curves]
	If $X$ is a generic plane curve of degree $d$, its Newton polygon is $d\cdot \Delta$ where $\Delta$ is the standard simplex. Furthermore, it is sufficiently generic with respect to its Newton polygon. By Corollary \ref{cor:newton}, the Newton polygon $\Delta_{X^*}$ has edges in directions $w_0=(-1,1)$, $w_1=(0,-1)$, and $w_2=(1,0)$, each of length
\[
\vol(\Delta_X)-\sigma_i/w_i=d^2-d.
\]
Hence, $\Delta_{X^*}=(d^2-d)\cdot \Delta$, and we recover the classical fact that a generic degree $d$ plane curve has dual of degree $d(d-1)$.
\end{ex}

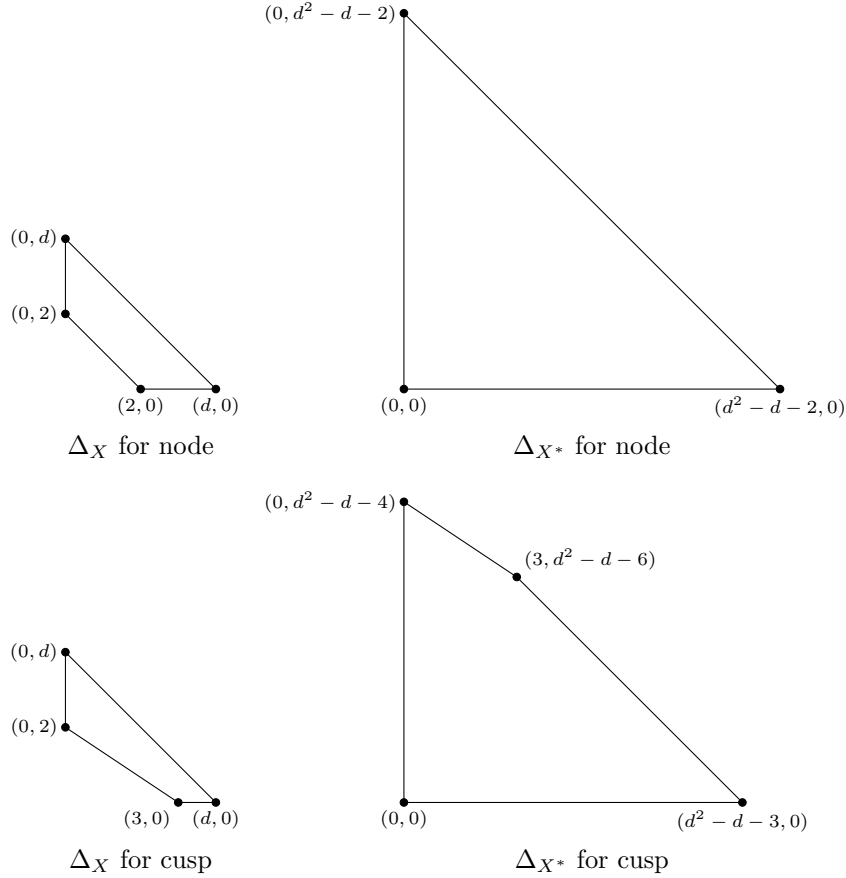
\begin{figure}
	\begin{tikzpicture}[scale=.5]
		\begin{scope}
		\draw (2,0) -- (0,2) -- (0,4) -- (4,0) -- (2,0);
	\draw[fill] (2,0) circle [radius=0.1];
	\draw[fill] (0,2) circle [radius=0.1];
	\draw[fill] (4,0) circle [radius=0.1];
	\draw[fill] (0,4) circle [radius=0.1];
	{\scriptsize
	\node [left] at (0,2) {$(0,2)$};
	\node [left] at (0,4) {$(0,d)$};
	\node [below] at (4,0) {$(d,0)$};
	\node [below] at (2,0) {$(2,0)$};}
	\node [below] at (2,-1) {$\Delta_{X}$ for node};
	\end{scope}
	\begin{scope}[shift={(9,0)}]
		\draw (10,0) -- (0,0) -- (0,10) -- (10,0);
	\draw[fill] (10,0) circle [radius=0.1];
	\draw[fill] (0,10) circle [radius=0.1];
	\draw[fill] (0,0) circle [radius=0.1];
	{\scriptsize
	\node [left] at (0,10) {$(0,d^2-d-2)$};
	\node [below] at (10,0) {$(d^2-d-2,0)$};
	\node [below] at (0,0) {$(0,0)$};}
	\node [below] at (5,-1) {$\Delta_{X^*}$ for node};
	\end{scope}
	\begin{scope}[shift={(0,-11)}]
		\draw (3,0) -- (0,2) -- (0,4) -- (4,0) -- (3,0);
	\draw[fill] (3,0) circle [radius=0.1];
	\draw[fill] (0,2) circle [radius=0.1];
	\draw[fill] (4,0) circle [radius=0.1];
	\draw[fill] (0,4) circle [radius=0.1];
	{\scriptsize
	\node [left] at (0,2) {$(0,2)$};
	\node [left] at (0,4) {$(0,d)$};
	\node [below] at (4,0) {$(d,0)$};
	\node [below left] at (3,0) {$(3,0)$};}
	\node [below] at (2,-1) {$\Delta_{X}$ for cusp};
	\end{scope}
	\begin{scope}[shift={(9,-11)}]
		\draw (0,0) -- (9,0) -- (3,6) --  (0,8) -- (0,0);
	\draw[fill] (9,0) circle [radius=0.1];
	\draw[fill] (3,6) circle [radius=0.1];
	\draw[fill] (0,8) circle [radius=0.1];
	\draw[fill] (0,0) circle [radius=0.1];
	{\scriptsize
	\node [left] at (0,8) {$(0,d^2-d-4)$};
	\node [above right] at (3,6) {$(3,d^2-d-6)$};
	\node [below] at (9,0) {$(d^2-d-3,0)$};
	\node [below] at (0,0) {$(0,0)$};}
	\node [below] at (5,-1) {$\Delta_{X^*}$ for cusp};
	\end{scope}
	\end{tikzpicture}
	\caption{Newton polygons for cusps and nodes}\label{fig:node}
\end{figure}
\begin{ex}[Nodes, cusps, and Pl\"uckers formula]\label{ex:pluecker}
Let $X$ be a generic plane curve of degree $d$ with an isolated nodal or cuspidal singularity at the origin. After changing coordinates, we can assume that its Newton polygon is
\[
\conv \{(2,0),(d,0),(0,d),(0,2)\}
\] 
in the nodal case or 
\[
\conv \{(3,0),(d,0),(0,d),(0,2)\}
\] 
in the cuspidal case. 
By Corollary \ref{cor:newton}, we obtain that  $\Delta_{X^*}$ has edges
\[
(d^2-d-2)(1,0),(d^2-d-2)(-1,1),(d^2-d-2)(0,-1)
\]
in the nodal case and
\[
(d^2-d-3)(1,0),(d^2-d-6)(-1,1),(-3,2),(d^2-d-4)(0,-1)
\]
in the cuspidal case. See Figure \ref{fig:node}.
Thus, in the nodal case, $\Delta_{X^*}=(d^2-d-2)\Delta$, and $X^*$ has degree $d(d-1)-2$. 
In the cuspidal case, the smallest dilate of $\Delta$ in which $\Delta_{X^*}$ fits is $(d^2-d-3)\Delta$, so 
$X^*$ has degree $d(d-1)-1$.

Combining this with the fact that the decrease to degree of $X^*$ from each singularity of $X$ is local (see e.g. \S 9 of \cite{brieskorn}, especially Lemma 3), we recover Pl\"ucker's classical formula that every node of $X$ lowers the degree of the dual curve by $2$, and every cusp by $3$.
\end{ex}

Generalizing the previous example, we can describe how a large class of singularities contribute to the decrease in degree of the dual curve:

\begin{lemma}
	\label{lemma:dd}
Let $X$ be an irreducible projective plane curve of degree $d$ which is generic with respect to its Newton polygon, and with  $Q=(0:0:1)$ the only singular point.
Let $A$ be the area of $d\Delta\setminus \Delta_X$. Let $\eta$ be the multiplicity of $Q$ in $X$, and $\tau$ the intersection multiplicity of the union of the coordinate axes with $X$ at $Q$. 
Then the degree of $X^*$ is $d^2-d-\delta$, where 
\[
\delta=A+\eta-\tau.
\]
\end{lemma}

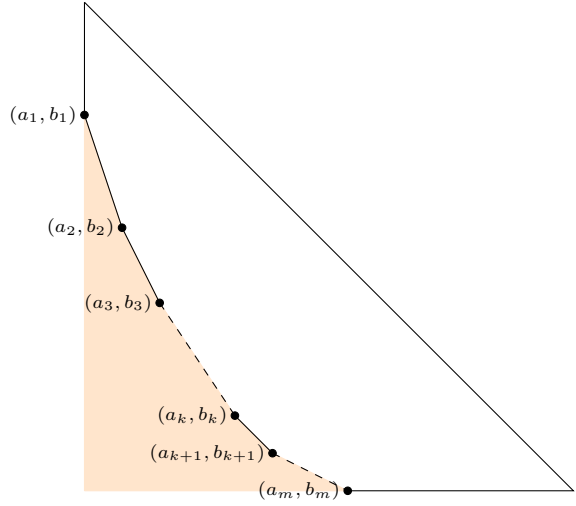
\begin{figure}
	\begin{tikzpicture}[scale=.5]
\draw [fill,color1] (0,0) -- (0,10) -- (1,7) -- (2,5) -- (4,2) -- (5,1) -- (7,0) -- (0,0);
\draw (0,13) -- (0,10) -- (1,7) -- (2,5);
\draw [dashed] (2,5) -- (4,2);
\draw (4,2) -- (5,1);
\draw [dashed] (5,1) -- (7,0);
\draw (7,0) -- (13,0) -- (0,13);
\draw[fill] (0,10) circle [radius=0.1];
\draw[fill] (1,7) circle [radius=0.1];
\draw[fill] (2,5) circle [radius=0.1];
\draw[fill] (4,2) circle [radius=0.1];
\draw[fill] (5,1) circle [radius=0.1];
\draw[fill] (7,0) circle [radius=0.1];
{\scriptsize
\node [left] at (0,10) {$(a_1,b_1)$};
\node [left] at (1,7) {$(a_2,b_2)$};
\node [left] at (2,5) {$(a_3,b_3)$};
\node [left] at (4,2) {$(a_k,b_k)$};
\node [left] at (5,1) {$(a_{k+1},b_{k+1})$};
\node [left] at (7,0) {$(a_m,b_m)$};
}\end{tikzpicture}

	\caption{The Newton polygon of an isolated singularity}\label{fig:deg}
\end{figure}

\begin{proof}
Consider the Newton polygon $\Delta_X$ of $X$; it has vertices $(d,0)$, $(0,d)$, and 
\[(a_1,b_1),\ldots, (a_m,b_m), \qquad 0=a_1<a_2<\ldots <a_m;\ b_1>b_2>\ldots>b_m=0.
\]
See Figure \ref{fig:deg} for an illustration.
Let $k$ be the index such that $(b_{i+1}-b_i)/(a_{i+1}-a_i)$ is smaller than $-1$ if and only if $i<k$.  Note that the area $A$ mentioned above is simply $d^2$ less the area of $\Delta_X$; it is the area of the shaded region in the figure. Using our description of $\Delta_{X^*}$ from Corollary \ref{cor:newton}, it is apparent that $(0,0)$ is a vertex of $\Delta_{X^*}$. Furthermore, the degree of $X^*$ is given by the length of the edge in direction $(1,0)$ of $\Delta_{X^*}$, plus $b_1-b_k+a_1-a_k$.
But this is the same as 
\begin{align*}
d^2-A-(d-a_m)+(b_1-b_k+a_1-a_k)\\
=d^2-d-(A-a_m-b_1+a_k+b_k).
\end{align*}
We conclude the proof by noting that $a_k+b_k$ is the multiplicity of $Q$ in $X$, and $a_m+b_1$ is the intersection multiplicity at $Q$ of the union of the coordinate axes with $X$.
\end{proof}

\begin{rem}
	Recall that the \emph{Newton polyhedron} of $X$ at $Q=(0:0:1)$ is the Minkowski sum of $\Delta_X$ with the positive orthant of $\RR^2$. Hence, the quantity $A$ in Lemma \ref{lemma:dd} may be interpreted as the area of the complement in $\RR_{\geq 0}^2$ of the Newton polyhedron of $X$ at $Q$.
\end{rem}

\begin{rem}\label{rem:df}
	Let $X$ be any integral plane curve of degree $d$.
	A generalization of Pl\"ucker's formula for the degree of the dual curve $X^*$ is the following formula due to Teissier \cite{lnm}:
\[
	\deg X^*=d(d-1)-\sum_{Q\in \sing X} (\mu(X,Q)+\eta(X,Q)-1).
\]
Here, the sum is taken over all singular points $Q$ of $X$, $\mu(X,Q)$ denotes the Milnor number of $(X,Q)$, and $\nu(X,Q)$ is the multiplicity of $X$ at $Q$. See also \cite{teissier:07a} for details.

This can be related to our Lemma \ref{lemma:dd} as follows. 	
Let $V(g)\subset \KK^2$ be an integral affine plane curve with singularity at the origin. Recall \cite{kouchnirenko} that the \emph{Newton number} of $V(g)$ at $0$ is 
\[A-\tau+1\]
where $A$ is the area of the complement of the Newton polyhedron of $g$ in the positive orthant, and $\tau$ is the intersection number of $V(g)$ with the sum of the coordinate axes. 
Kouchnirenko has shown that as long as $g$ is non-degenerate with respect to its Newton boundary, the Newton number is equal to $\mu(V(g),0)$, the Milnor number of $V(g)$ at $0$ \cite[Theorem II]{kouchnirenko}.

We will say that a plane curve singularity $(X,Q)$ is \emph{sufficiently generic} if there is some plane curve $X'$ with single singularity $Q'$ such that the germs $(X,Q)$ and $(X',Q')$ are isomorphic, and $X'$ is sufficiently generic with respect to its Newton polygon for Corollary \ref{cor:newton} to apply.
Suppose that our curve $X$ is such that each singularity $(X,Q)$ is sufficiently generic, and is non-degenerate with respect to the Newton boundary. Lemma \ref{lemma:dd} implies that each singularity $Q$ of $X$ decreases the degree of $X^*$ by $\delta=A+\eta-\tau$. The quantity $A-\tau+1$ is the Newton number of $(X,Q)$, so by Kouchnirenko's result, it follows that each singularity $Q$ decreases the degree of $X^*$ by $\mu(X,Q)-1+\eta(X,Q)$. We thus recover Teissier's formula, assuming that the singularities of $X$ are sufficiently generic and non-degenerate with respect to the Newton boundary.
\end{rem}

\section{Surfaces}\label{sec:surfaces}
\subsection{Setup and simplest case}\label{face:1}

We now work towards a description of $\trop(X^*)$ when $X$ is a surface in $\PP^3$ with a smooth tropicalization. By Theorems \ref{thm:vert} and \ref{thm:edge}, we  only have to deal with points contained in the relative interior of a two-face $F$. For such a point $p$, the corresponding matrix $M^{(0)}$ will only have two columns. This means that there is a unique $M^{(0)}$-consistent set $J_0$: it will consist of $\{-1\}$ and all indices $i$ such that $M_i^{(0)}$ is proportional to $(1,1)$. By Lemma \ref{lemma:geco}, this set $J_0$ will be of the form $\{-1\}\cup J$, where $J$ is the maximal subset of $\{0,\ldots,n\}$ such that $\langle J\rangle$ is contained in $\langle F \rangle$. There are three cases, which we deal with separately: $\#J=0$, $\#J=1$, and $\#J=2$.

The case $\#J=0$ is exactly the situation of Proposition \ref{prop:face}, which we now prove:
\begin{proof}[Proof of Proposition \ref{prop:face}]
	Since $J_0=\{-1\}$, 
	Proposition \ref{prop:consistency} implies that the only possible tropical tangent to $p$ is $-p$. This is in fact a tropical tangent:	
	the maximal consistent sequence $(J_0)$ satisfies the lifting criteria with $L=\{0\}$, and $I_1=\emptyset$, so Proposition \ref{prop:lifting} applies. This guarantees that $-p$ is a tropical tangent to $p$ as desired.
\end{proof}

\subsection{Face in one standard direction}\label{face:2}
The next case to consider is when $\#J=1$. This is exactly the situation of Proposition \ref{prop:face2}. Here, we have fixed a two-face $F$ with $e_i$ contained in $\langle F \rangle$. 
For any $p$ in the relative interior of our two-face $F$, let $\phi(p)=\nu_1(p)-\nu_0(p)$. Then $\phi$ is a piecewise linear concave function.

Fix $p$ in the relative interior of $F$.
\begin{lemma}\label{lemma:nonempty}
The set $\A_1$ is not empty.
\end{lemma}
\begin{proof}
	Since $e_i\in \langle F \rangle$, every element $u\in \A_0$ has the same $i$th coordinate, say $\alpha$. If $\A_1$ is empty, then every $u\notin \A_0$ must also have $\alpha$ as its $i$th coordinate. Then $\A$ is contained the hyperplane where all $i$th coordinates are equal to $\alpha$ which implies that either $X$ is not irreducible, or $X$ is a cone. 
\end{proof}

\begin{lemma}\label{lemma:prep}
Suppose that $\trop(X)$ has generic valuations.
	The closure of the set of tropical tangents to $p$ are
\[
-p-s_ie_i
\]
for any $s_i\geq \phi(p)$, with $s_i=\phi(p)$ unless $\#\A_1>1$.
\end{lemma}
\begin{proof}
Since $J_0=\{-1,i\}$, we have $I_1=\{i\}$. There is a unique maximal consistent $J_1$: this will be $J_1=\emptyset$ if $\#\A_1=1$ and $J_1=\{i\}$ if $\#\A_1>1$. By Proposition \ref{prop:consistency}, any tropical tangent at $p$ must have the form claimed.

On the other hand, we claim that the maximal consistent sequence $(J_0,J_1)$ satisfies the lifting criteria for $L=\emptyset$. Indeed, if $J_1=\emptyset$ the second and third criteria are straightforward to verify. If instead $J_1=\{i\}$, the second criteria is still straightforward. For the third criterion, we must show that for any two elements $u\neq v\in \A_1$, $u-v$ is not proportional to the difference of the two distinct elements of $\A_0$. 
If $u-v$ is proportional to the difference of the two distinct elements of $\A_0$, then 
\[
0=\nu(c_u)-\langle q,u\rangle - (\nu(c_v)-\langle q,v\rangle)
\]
for all $q\in F$. But this contradicts $\trop(X)$ having generic valuations. We conclude that $(J_0,J_1)$ satisfies the lifting criteria.

Note that $z_{i}^{(1)}=z_i$, so Proposition \ref{prop:lifting} guarantees the existence of all tangents with $s_i=\phi(p)$ (if $\#\A_1=1$) $s_i>\phi(p)$ (if $\#\A_1>1$). Taking the closure, we obtain the inequality $s_i \geq \phi(p)$ in the latter situation.
\end{proof}

In order to complete the proof of Proposition \ref{prop:face2}, we will need to interpret $\phi$ in terms of $\delta$.
The difficulty is that since $\dim F=2$, Proposition \ref{prop:a1} does not apply verbatim.  However, Lemmas \ref{lemma:a1} and \ref{lemma:distance} may be applied to conclude that $\A_1'$ (as opposed to $\A_1$) consists of those $u\in \A$ corresponding to facets $Q$ of $F$ whose lattice distance is closest to $p$. When is an element  $u$ of $\A_1'$ actually in $\A_1$? This will be the case if and only if $u_i^{(1)}\neq 0$, or equivalently, the facet $Q$ of $F$ corresponding to $u$ does not have $\langle Q \rangle=\langle e_i \rangle$.
\begin{proof}[Proof of Proposition \ref{prop:face2}]
	To start, notice that we cannot have $E_S=\emptyset$, since in that case, $\A_0=\A$, contradicting Lemma \ref{lemma:nonempty}.
Suppose first that not every element of $E_S$ is parallel to $e_i$. Then the set $\A_1$ consists of those $u\in \A$ corresponding to the elements of $E_S$ not parallel to $e_i$, and for $p$ in the relative interior of $S$ it follows that $\phi(p)=\dd(p)$. 
The claim now follows from Lemma \ref{lemma:prep}.

Next, we assume that $E_S$ consists only of elements parallel to $e_i$ and that the recession cone of $S$ doesn't contain $e_i$ or $-e_i$. Then for any point $q$ in the relative interior of $S$, the line segment $(q+\langle e_i \rangle )\cap S$ has two distinct endpoints on which $\phi$ takes the same values as $\dd$. Furthermore, since the endpoints of this segment correspond to different sets $\A_1$, there must be some point $q'$ on this segment for which $\#\A_1(q')>1$. Together with the concavity of $\phi$, Lemma \ref{lemma:prep} implies that the closure of the tropical tangents along this segment are exactly
\begin{equation}\label{eqn:eee}
	\{-p-\dd(p)e_i\ |\ p\in (q+\langle e_i \rangle )\cap S  \}-\RR_{\geq 0}\cdot e_i.
\end{equation}

If instead the recession cone of $S$ contains $e_i$ (but not $-e_i$), then for any point $q$ in the relative interior of $S$, the ray $(q+\langle e_i \rangle )\cap S$ has exactly one endpoint, on which $\phi$ takes the same values as $\dd$. Since $\phi\geq 0$ and this ray has recession cone $e_i$, again Lemma \ref{lemma:prep} implies that the closure of the tropical tangents along this ray are exactly as in \eqref{eqn:eee}.
Letting the point $q$ vary within the relative interior of $S$ in both of these cases, we obtain the description of $S^*$ from the proposition.

Finally, suppose that $-e_i$ is in the recession cone of $F$. We cannot have $e_i$ in the recession cone of $F$, otherwise the Newton polytope $\Delta_X$ will be contained in a hyperplane orthogonal to $e_i$, implying that $X$  is reducible or a cone.
For $u\in \A_0$, $u_i$ is some constant, say $\lambda$. For the $u\in \A$ corresponding to any face of $F$ parallel to $e_i$, we also have $u_i=\lambda$. Furthermore, since $-e_i$ is in the recession cone of $F$, we must have $u_i\geq \lambda$ for all $u\in \A$. We conclude that $\lambda=0$, since otherwise $X$ is not irreducible.

For any $u\in \A$ corresponding to an edge of $F$ not parallel to $e_i$, we must have $u_i=1$, otherwise $\trop(X)$ is not smooth. Furthermore, any $u\in \A_1$ must have $u_i\geq 1$. For $p$ in the relative interior of $F$, consider the ray $(p+\langle e_i \rangle )\cap F$. Near its endpoint $q$, $\A_1$ consists only of some set of $u\in \A$ which correspond to edges of $F$. It follows that  $\A_1(q-\lambda e_i)$ is constant for any $\lambda >0$, since for $u\in \A_1$ near $q$, $u_i=1$, and as $\lambda$ varies, only the $i$th coordinate matters.
This implies that $-p-\phi(p)e_i=-q$. 

It remains to see when $\#\A_1(p)>1$. But this is exactly when $q$ is a vertex of $F$ not adjacent to any edge of $F$ parallel to $e_i$. Lemma \ref{lemma:prep} now 
implies the third claim of the proposition. 
\end{proof}

\subsection{Face in two standard directions}\label{face:3}
We finally arrive in the case $\#J=2$. This is the situation of Proposition \ref{prop:face3}. Here, we have fixed a two-face $F$ with $\langle F \rangle=\langle e_i,e_j\rangle$ for some $i\neq j$. Then $J=\{i,j\}$, and $J_0=\{-1,i,j\}$, from which follows that $I_1=\{i,j\}$. It is straightforward to understand $\A_1$ and $\nu_1-\nu_0$: since $\#I_1=\dim F=2$, Proposition \ref{prop:a1} applies, so $\A_1$ consists of some $u\in \A$ corresponding to facets of $F$, and $\nu_1-\nu_0=\dd$.

Since $\nu_1-\nu_0=\dd$ and elements of $\A_1$ always correspond to facets of $F$, the set $\A_1$ is constant on the relative interior of any $S\in\cS_F$. Furthermore, since $\trop(X)$ has generic valuations, for any $S\in \cS_F$, $\dim S=3-\#\A_1=3-\#E_S$. 
By Proposition \ref{prop:consistency}, for $p$ in the relative interior of $S$, any tropical tangent must belong to 
\begin{equation}\label{eqn:tan}
-p-\dd(p)(e_i+e_j)-\sum_{\ell\in J_1}\RR_{\geq 0}\cdot e_\ell
\end{equation}
with $J_1$ an $M^{(1)}$-consistent set. In some cases, we will show that this set is the closure of all tropical tangents; in other cases, we will derive additional conditions.

\begin{proof}[Proof of Proposition \ref{prop:face3}]
First we suppose that $\dim S=0$, which implies that $\#E_S=\#\A_1=3$. Two columns in $M^{(1)}$ are proportional if and only if the corresponding edges in $E_S$ are parallel. It follows that the $2\times 3$ matrix $M^{(1)}$ must have full rank two. It is straightforward to check $\{i\}$ (or $\{j\}$) is consistent if and only if no two edges of $E_S$ are parallel to $e_i$ (or $e_j$, respectively). Likewise, $\{i,j\}$ is consistent if and only if no two edges of $E_S$ are parallel. 
In all cases, taking $J_1$ to be a maximal consistent set leads to a sequence $(J_0,J_1)$ which satisfies the lifting conditions for $L=\emptyset$. Indeed, this follows from the fact that each element of $\A_1$ belongs to a simplex in the triangulation of $\Delta_X$ containing $\A_0$. Thus, after possibly taking the Euclidean closure to account for non-maximal consistent sequences, the set of all possible tropical tangents is exactly that of \eqref{eqn:tan}. To summarize, these are exactly the tangents of the form 
\begin{equation}\label{eqn:tan2}
-p-\dd(p)(e_i+e_j)-\lambda_i\cdot e_i-\lambda_j\cdot e_j
\end{equation}
for $\lambda_i,\lambda_j\geq 0$ with at most one of $\lambda_i,\lambda_j$ non-zero if $E_S$ contains any parallel edges, and $\lambda_\ell=0$ if two edges of $E_S$ are parallel to $e_\ell$.

Next we suppose that $\dim S=1$, which implies $\#E_S=\#\A_1=2$. The rank of $M^{(1)}$ is two if and only if the two edges in $E_S$ are not parallel. Otherwise, the rank of $M^{(1)}$ is one. In the hypotheses of the proposition, we have assumed that the two edges in $E_S$ 
are not parallel. Since $M^{(1)}$ is a full rank $2\times 2$ matrix, $\{i,j\}$ cannot be an $M^{(1)}$-consistent set. On the other hand, it is straightforward to check that $\{i\}$ (or $\{j\}$) is consistent if and only if no edge of $E_S$ is parallel to $e_i$ (or $e_j$, respectively).
Any $(J_0,J_1)$ with $J_1$ maximal consistent satisfies the lifting criteria for $L=\emptyset$.
 Thus, after possibly taking the Euclidean closure to account for non-maximal consistent sequences, the set of all possible tropical tangents is exactly that of \eqref{eqn:tan}. 
We conclude that in this situation, the tropical tangents are exactly of the form
of \eqref{eqn:tan2}
for $\lambda_i,\lambda_j\geq 0$ with at most one of $\lambda_i,\lambda_j$ non-zero, and $\lambda_\ell=0$ if one edge of $E_S$ is parallel to $e_\ell$.

Finally, we suppose that $\dim S=2$, and the edge $E$ of $E_S$ is not parallel to $e_i$ or $e_j$. Then
$J_1=\emptyset$ is the unique $M^{(1)}$-consistent set, so $(J_0,J_1)$ is the unique consistent sequence. This satisfies the lifting criteria for $L=\emptyset$, so everything appearing in \eqref{eqn:tan} is a tropical tangent. Since however, $J_1=\emptyset$, there is only one tropical tangent: the tangent of \eqref{eqn:tan2} for $\lambda_i=\lambda_j=0$.

Putting together the above analysis, we arrive at the claim of Proposition \ref{prop:face3}.
\end{proof}

We will prove Propositions \ref{prop:face5} and \ref{prop:face4} together. These situations are more difficult, because the consistent sequences involved have a $J_2$ term.
We will assume that $i<j$.
For a fixed point $p$ in the relative interior of $S$, we will initially split things up into six cases, which are differentiated by the dimension of $S$ ($1$ or $2$), the behaviour of the edges in $E_S$, and the size of $\#\A_2$. We list them in the table below, omitting some cases that may be obtained by interchanging $i$ and $j$.
For each case, we list the unique maximal choice of $J_1$, the resulting $K_1$, $I_2$, and $J_2$, and then a choice of the set $L$ to be used for verifying the lifting criteria:

\vspace{.5cm}
\begin{center}
\begin{tabular}{|l|l| c| c c c c |c|}
	\cline{3-8}
	\multicolumn{2}{c|}{} &$\#\A_2$& $J_1$ & $K_1$ & $I_2$ & $J_2$ & $L$\\
	\hline
A.0&	$\dim S=1$, parallel edges in $E_S$,& $0$ & $\{i,j\}$ & $\{i\}$ & $\{j\}$ & $\emptyset$ & $\emptyset$\\
\cline{1-1}\cline{3-8}
A.1& not parallel to $e_i,e_j$ & $1$ & $\{i,j\}$ & $\{i\}$ & $\{j\}$ & $\emptyset$ & $\{j\}$\\
\cline{1-1}\cline{3-8}
A.2&		& $2$ & $\{i,j\}$ & $\{i\}$ & $\{j\}$ & $\{j\}$ & $\emptyset$\\

\hline
B.1&\multirow{2}{*}{$\dim S=1$, edges in $E_S$ parallel to $e_j$} & $1$ & $\{i,j\}$ & $\{i\}$ &  $\{j\}$& $\emptyset$&$\emptyset$\\
\cline{1-1}\cline{3-8}
B.2&& $2$ & $\{i,j\}$ & $\{i\}$ & $\{j\}$ & $\{j\}$&$\emptyset$\\
\hline
C.1&\multirow{2}{*}{$\dim S=2$, edge in $E_S$ parallel to $e_j$} & $1$ & $\{j\}$ & $\emptyset$ &  $\{j\}$& $\emptyset$&$\emptyset$\\
\cline{1-1}\cline{3-8}
C.2&& $2$ & $\{j\}$ & $\emptyset$ & $\{j\}$ & $\{j\}$&$\emptyset$\\
\hline
\end{tabular}
\end{center}
\vspace{.5cm}

In cases A.0, A.1 and A.2, $z_j^{(2)}=z_j-\alpha z_i$ for some non-zero $\alpha$. In all other cases, $z_j^{(2)}=z_j$. Applying Proposition \ref{prop:consistency}, we see that any tropical tangent to $p$ must be of the form
\[
-p-\dd(p)(e_i+e_j)-\lambda_i\cdot e_i-\lambda_j\cdot e_j
\]
for $\lambda_i,\lambda_j\geq 0$ satisfying the following constraints:
\begin{enumerate}[({Case} A)]
	\item Either $0\leq \lambda_i=\lambda_j\leq \nu_2-\nu_1$, or $\lambda_i,\lambda_j\geq \nu_2-\nu_1$ with either $\lambda_i=\nu_2-\nu_1$ or $\lambda_j=\nu_2-\nu_1$ unless $\#\A_2=2$. Note that if $\A_2=\emptyset$, we use the convention that $\nu_2=+\infty$, which implies that $\lambda_i=\lambda_j$.\label{casea}
	\item  $\lambda_j\geq \nu_2-\nu_1$ with equality unless $\#\A_2=2$. \label{caseb}
	\item  $\lambda_i=0$, and $\lambda_j\geq \nu_2-\nu_1$ with equality unless $\#\A_2=2$. \label{casec}
 \end{enumerate}
The closure for the set of tropical tangents to $p$ is in fact this entire set, since we may again apply Proposition \ref{prop:lifting}. The maximal consistent sequences $(J_0,J_1,J_2)$ satisfy the lifting criteria for our choice of $L$ in every case; this follows from $\trop(X)$ having generic valuations.
Note that in case \ref{casea}, we may choose $z_i=z_i^{(1)}$ and $z_j^{(2)}$ as we wish (subject to conditions on their valuations), which also allows us to choose $z_j$ as we wish. For case \ref{casea}.0, we have $z_j^{(2)}=0$, which implies that $z_j$ is already determined from $z_i$. 

It remains to check that in each case, as $p$ ranges over the relative interior of $S$, the closure of the set of tropical tangents equals the claimed set $S^*$ from the proposition. 

\begin{proof}[Proof of Proposition \ref{prop:face4}]
 If $S$ has no endpoints, then we must have $\A_2=\emptyset$, so we are in case \ref{casea}.0, and we obtain that $S^*$ consists of $-p-\lambda_ie_i-\lambda_je_j$ for $\lambda_i=\lambda_j\geq 0$ and $p\in S$. We will subsequently assume that $S$ has endpoints.

Near an endpoint $q$ of $S$, $\A_2$ consists of the element $u\in \A$ corresponding to the edge of $E_q$ not in $E_S$. This implies that as $p$ approaches the endpoints of $S$, $\nu_2-\nu_1$ approaches $0$.
If $S$ is bounded and hence has two distinct endpoints, then the set $\A_2$ must change over the course of $S$, so at some point $p$ in the interior of $S$, $\A_2(p)$ consists of two elements. Using our analysis of situation A and B above, a straightforward convexity argument shows that if $S$ is bounded, we obtain exactly the set claimed in the proposition.

Assume instead that $S=q+\RR_{\geq 0}\cdot v$ for $v=\alpha_ie_i+\alpha_je_j$. 
Then there is some $p$ in the interior of $S$ for which $\#\A_2(p)=2$ if and only if $S$ is purely primitive. Indeed, since $\trop(X)$ is smooth, there is \emph{some} edge $E$ of $F$ such that $\dd_E(q+v)-\dd_E(q)=1$. Furthermore, $(\nu_2(q+v)-\nu_0(q+v))-(\nu_2(q)-\nu_0(q))\geq 1$, so on all but a bounded portion of $S$, $\nu_2-\nu_0$ equals some such $\dd_E$. If $S$ is purely primitive, then $\nu_2-\nu_0=\dd_E$ all along $S$, otherwise $\nu_2-\nu_0$ is not linear, so there must be some point for which $\#\A_2>1$.

We must now differentiate based on the direction of $v$. Take for instance $\alpha_i=\alpha_j=-1$. Then $-p-\dd(p)(e_i+e_j)-(\nu_2-\nu_1)(e_i+e_j)=-q-\dd(q)(e_1+e_2)-\beta(e_i+e_j)$ for some $\beta\geq 0$ for all $p\in S$, with $\beta=0$ if $S$ is purely primitive. If $S$ is purely primitive, we are always in case \ref{casea}.1, and obtain exactly the tangents $-p-\dd(p)(e_i+e_j)$ for $p\in S$ and $-q-\dd(q)(e_i+e_j)-\lambda_ie_i-\lambda_je_j$ with $\lambda_i\lambda_j=0$. If $S$ is not purely primitive, then for some $p$ we are in case \ref{casea}.2, and we may drop the condition $\lambda_i\lambda_j=0$.

If instead $\alpha_i=0$ and $\alpha_j=-1$, then $-p-\dd(p)(e_i+e_j)-(\nu_2-\nu_1)\cdot e_j=-q-\dd(q)(e_1+e_2)-\beta e_j$ for some $\beta\geq 0$ for all $p\in S$, with $\beta=0$ if $S$ is purely primitive. If $S$ is purely primitive, we are always in case \ref{caseb}.1, and obtain exactly the tangents $-q-\dd(q)(e_i+e_j)-\lambda_ie_i$. If $S$ is not purely primitive, then for some $p$ we are in case \ref{caseb}.2, and we will 
obtain tangents $-q-\dd(q)(e_i+e_j)-\lambda_ie_i-\lambda_je_j$. A similar analysis holds if $\alpha_i=-1$ and $\alpha_j=0$. 

For the remaining directions of $v$ we are always in case \ref{casea}. Let 
\[\phi_1(p)=-p-\dd(p)(e_i+e_j)\] and \[\phi_2(p)=-p-(\nu_2-\nu_0)(e_i+e_j).\] 
Then $S^*$ consists of  
\[\conv \{\phi_k(p)\ |\ p\in S,\ k=1,2\},\]
along with
\[\phi_2(p)-\lambda_i e_i-\lambda_j e_j\]
for all $p\in S$, with $\lambda_i,\lambda_j\geq 0$ and $\lambda_i\lambda_j=0$ for given $p$ unless $\phi_2$ is non-linear at $p$.
 Suppose for example that $\alpha_i=-1$ and $\alpha_j<-1$.
 Then the ray $\phi_2(q+\gamma v)$ for $\gamma$ sufficiently large is parallel to $e_j$. In fact, if $S$ is purely primitive, then this is true for all $\gamma\geq 0$, and the ray is the ray in direction $e_j$ with endpoint $-q-\dd(q)(e_i+e_j)$. The description of $S^*$ in the proposition follows in this case. If $S$ is not purely primitive, then since for for $p$ in $S$ we obtain $\phi_2(p)-\RR_{\geq 0}e_i-\RR_{\geq 0}e_j$, the description also follows in this case. 
For the remaining directions of $v$, a similar straightforward analysis always leads to the set $S^*$ claimed in the proposition. We leave the details to the reader. 
\end{proof}
\begin{proof}[Proof of Proposition \ref{prop:face5}]
We now  find ourselves in case \ref{casec} from above.
First, assume that $E$ is bounded. 
 Then for any point $q$ in the relative interior of $S$, the line segment $S_q=(q+\langle e_i \rangle )\cap S$ has two distinct endpoints. Near an endpoint $v$ of $S_q$, $\A_2$ consists of the element $u\in \A$ corresponding to an edge $Q$ of $F$ for which $\dd_Q(v)=\dd$. This implies that as $p$ approaches the endpoints of $S$, $\nu_2-\nu_1$ approaches $0$. Furthermore, the set $\A_2$ must change over the course of $S_q$, so at some point in the interior of $S$, $\A_2$ consists of two elements. A convexity argument shows that the closure of the set of tropical tangents agrees with the $S^*$ from the proposition statement. 

If $E$ is unbounded but $-e_j$ is not in the recession cone of $F$, then $E$ must be unbounded in direction $e_j$. It is no longer the case that we necessarily find points for which $\A_2$ consists of more than one element. Nonetheless, since $S$ is unbounded in direction $e_j$, a convexity argument still shows that we get the desired set $S^*$. 

Finally, suppose that $-e_j$ is in the recession cone of $F$. 
Similar to at the end of the proof of Proposition \ref{prop:face2}, we must have that for $u\in \A_0$, $u_j=0$, as well as for those $u\in \A$ corresponding to any face of $F$ parallel to $e_j$.
For any $u\in \A$ corresponding to an edge of $F$ not parallel to $e_j$, we must have $u_j=1$, otherwise $\trop(X)$ is not smooth. Furthermore, any $u\in \A_2$ must have $u_j\geq 1$. For $p$ in the relative interior of $F$, consider the ray $(p+\langle e_j \rangle )\cap F$.

Near its endpoint $q$, $\A_2$ consists only of some set of $u\in \A$ which correspond to edges of $F$. It follows that  $\A_2(q-\lambda e_j)$ is constant for any $\lambda >0$, since for $u\in \A_2$ near $q$, $u_i=1$, and as $\lambda$ varies, only the $i$th coordinate matters.
This implies that $-p-\dd(p)(e_i+e_j)-(\nu_2-\nu_1)=-q-\dd(q)(e_i+e_j)$. 

It remains to see when $\#\A_2(p)>1$. But this is exactly when $q$ is an interior vertex of $S'$.
The claim of the proposition now follows from the description of tangents in case \ref{casec} above.
\end{proof}

\subsection{Putting it all together}\label{face:4}
Our analysis of tropical tangents to points $p\in \trop(X)$ yields a complete description of $\trop(X^*)$:

\begin{proof}[Proof of Theorem \ref{thm:surface}]
	We argue as in the proof of Theorem \ref{thm:dualcurve} that the closure of the union of all tropical tangents is $\trop(X^*)$.
	Indeed, the only points in $\trop(X^*)$ we might miss are  those which are tangent at points in the boundary of $X$. But these will form a set in $\trop(X^*)$ of codimension at least one. Removing it from $\trop(X^*)$ and taking the closure again yields $\trop(X^*)$.
\end{proof}

We now finally show Theorem \ref{thm:face}:
\begin{proof}[Proof of Theorem \ref{thm:face}]
If $J=\emptyset$, then the claim follows directly from Proposition \ref{prop:face}.
If $J=\{i\}$, then we are in the situation of Proposition \ref{prop:face2}, and the hypothesis of the theorem guarantees that $-e_i$ is not in the recession cone of $F$. Applying the first two parts of the proposition, it is straightforward to verify that $F^*=-F-\RR_{\geq 0}\cdot e_i$.

Finally, suppose that $J=\{i,j\}$. We are now in the situations of Propositions \ref{prop:face3}, \ref{prop:face5}, and \ref{prop:face4}. By assumption, no non-zero element of the positive hull of $-e_i$ and $-e_j$ is in the recession cone of $F$. This means in particular that we may exclude the third case of Proposition \ref{prop:face3}, the second case of Proposition \ref{prop:face5}, and most cases of Proposition \ref{prop:face4}. 

It is straightforward to check that for $p$ in the boundary of $F$, $-p$ is in $F^*$. Furthermore, 
using a convexity argument and the assumption on the recession cone of $F$, it follows that $-F$ is in $F^*$. 
Consider any point $-q$ in the interior of $-F$. We claim that $-q-\RR_{\geq 0}\cdot e_i\in F^*$. Indeed, if $F$ has $e_i$ in its recession cone, the claim is immediate. Otherwise, let $-p_1$ and $-p_2$ be the endpoints of the segment $(-q+\langle e_i \rangle)\cap -F$. Then again by the assumption on the recession cone of $F$ and the fact that $\dd(p)=0$ on the boundary of $F$, there is a piecewise linear path $\gamma:[0,1]\to F$  with endpoints $p_1$ and $p_2$ such that 
\[
-\gamma(t)-\dd(\gamma(t))(e_i+e_j)
\]
is contained in $(-q+\langle e_i \rangle)$ for any $t\in [0,1]$. At some point, $\gamma$ must cross a one dimensional cell $S$ of the subdivision $\cS_F$ induced by $\dd$, say at the point $p$. If the edges of $E_S$ are parallel, Proposition \ref{prop:face4} applies, and we obtain that 
\[
	-p-\dd(p)(e_i+e_j)-\RR_{\geq 0}\cdot e_i
\]
is contained in $F^*$, but $-p-\dd(p)$ lies in $(-q+\langle e_i \rangle)$, and the claim follows. If instead the edges of $E_S$ are not parallel, Proposition \ref{prop:face3} applies. The claim again follows, as long as one edge of $E_S$ is not parallel to $e_i$. But in this case, we may move into the neighboring two-dimensional cell $S'$ whose edge is parallel to $e_i$ to again obtain the desired tropical tangents (using Proposition \ref{prop:face5}).

A similar argument shows that $-q-\RR_{\geq 0}\cdot e_j\in F^*$. Now, to show that
\[F^*=-F-\RR_{\geq 0}\cdot e_i-\RR_{\geq 0}\cdot e_j\]
we again consider several cases. If $F$ is unbounded, the claim already follows from the above discussion, so we may assume that $F$ is bounded. This implies that $\cS_F$ must contain a zero-dimensional cell $S$. If no two edges of $E_S$ are parallel, then Proposition \ref{prop:face3} tells us that 
\begin{equation}\label{lasteq}
	-q-\dd(q)(e_i+e_j)-\RR_{\geq 0}\cdot e_i-\RR_{\geq 0}\cdot e_j
\end{equation}
is contained in $F^*$ for $q=S$, which together with the above discussion implies the claim.

Suppose instead that two edges of $E_S$ are parallel. Then moving to the neighboring one-dimensional cell $S'$ with the two edges of $E_{S'}$ parallel, we may apply Proposition \ref{prop:face5}. Since $F$ is bounded, we again obtain tangents as in \eqref{lasteq} for a point $q$ arbitrarily close to $S$. Again, the claim follows. 
\end{proof}

\section{Future Directions}\label{sec:conclusion}

\subsection{Tropical projective duality, and singular curves}
One of the striking features of projective duality is that it is an involution: given $X\subset \PP^n$, we have $(X^*)^*=X$. It is tempting to try to extend our descriptions of tropical dual curves and surfaces to give an involution on the set of tropical plane curves, and on the set of tropical surfaces in $\RR^3$. Here, we will focus on the situation of plane curves.

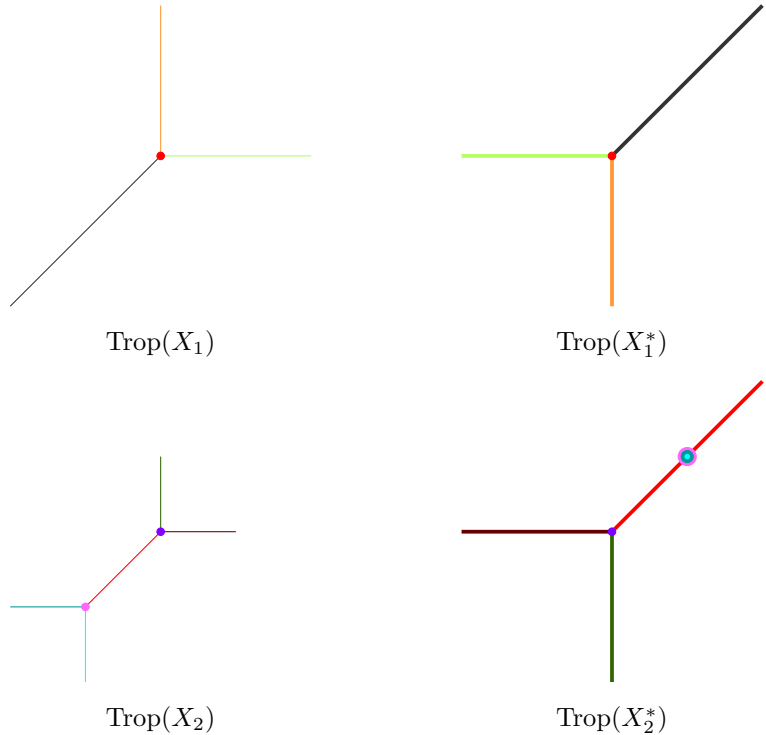
\begin{figure}
\begin{tikzpicture}

\begin{scope}
\draw[color1a] (-2,-2) -- (0,0);
\draw[color2] (0,0) -- (2,0);
\draw[color3] (0,0) -- (0,2);

\draw[fill, color4] (0,0) circle [radius=0.05];
\node at (0,-2.5) {$\trop(X_1)$};
\end{scope}

\begin{scope}[shift={(6,0)}]
\draw[line width=.5mm,color1a] (2,2) -- (0,0);
\draw[line width=.5mm,color2] (0,0) -- (-2,0);
\draw[line width=.5mm,color3] (0,0) -- (0,-2);

\draw[fill, color4] (0,0) circle [radius=0.05];
\node at (0,-2.5) {$\trop(X_1^*)$};
\end{scope}

\begin{scope}[shift={(0,-5)}]
\draw[color4] (-1,-1) -- (0,0);
\draw[color5] (0,0) -- (1,0);
\draw[color6] (0,0) -- (0,1);
\draw[color7] (-1,-1) -- (-2,-1);
\draw[color8] (-1,-1) -- (-1,-2);

\draw[fill, color9] (0,0) circle [radius=0.05];
\draw[fill, color10] (-1,-1) circle [radius=0.05];
\node at (0,-2.5) {$\trop(X_2)$};
\end{scope}
\begin{scope}[shift={(6,-5)}]
\draw[line width=.5mm,color4] (2,2) -- (0,0);
\draw[line width=.5mm,color5] (0,0) -- (-2,0);
\draw[line width=.5mm, color6] (0,0) -- (0,-2);

\draw[fill, color9] (0,0) circle [radius=0.05];
\draw[fill, color10] (1,1) circle [radius=.12];
\draw[fill, color7] (1,1) circle [radius=0.08];
\draw[fill, color8] (1,1) circle [radius=0.03];
\node at (0,-2.5) {$\trop(X_2^*)$};
\end{scope}
\end{tikzpicture}
\caption{Tropical quadrics with the same tropical dual curve}\label{fig:curvex}
\end{figure}

The first difficulty we face is that the map that takes a smooth tropical curve $\trop(X)$ to its dual $\trop(X^*)$ is not injective: both tropical quadrics $\trop(X_1)$ and $\trop(X_2)$ pictured in the left of Figure \ref{fig:curvex} have the same dual, pictured on the right. As in Example \ref{ex:dualCurveExample}, thick lines represent edges of multiplicity two, and the colors on the right correspond to the contributing pieces on the left. The issue here is that the rays emanating from the bottom left vertex of $\trop(X_2)$ both are contracted to points in $\trop(X_2^*)$. 

To remedy this non-injectivity, we may decorate $\trop(X^*)$ by marking each point $-p$ of $\trop(X^*)$ for which $p$ is a vertex of $\trop(X)$ with edges emanating in directions $-e_0$, $-e_1$, and $-e_2$. We call $\trop(X^*)$ with these markings the \emph{decorated tropical dual} of $\trop(X)$.
\begin{prop}\label{prop:injectivity}
The map which assigns to any tropical curve $\trop(X)$ its decorated tropical dual is injective. In other words, we may recover any smooth tropical curve $\trop(X)$ from its decorated tropical dual.
\end{prop}
\begin{proof}
The claim of the proposition follows by analyzing the local contribution to $\trop(X^*)$ of the portion of $\trop(X)$ surrounding any vertex $p$. Here, we differentiate among different cases based on the number of standard directions $e_0,e_1,e_2$ emanating from $p$, the number of anti-standard directions $-e_0,-e_1,-e_2$, and the number of non-standard directions.
Figure \ref{figure:local} lists the $5$ cases which occur, and gives an example of what $\trop(X)$ and $\trop(X^*)$ may look like in a neighborhood of the vertex. As before, thick edges denote edges of multiplicity two. The circle in the final case denotes the marking on the decorated tropical dual.

We observe that the vertex $-p$ of $\trop(X^*)$ is different in each of the five cases. Also observe that none of these vertices has a line passing through it.
Hence, to recover $\trop(X)$ from $\trop(X^*)$, we may proceed as follows: 
\begin{enumerate}
	\item For each vertex $v$ of $\trop(X^*)$, remove any lines through the vertex;
	\item Match the remaining local picture around $v$ to one of the five dual vertex types;
	\item Recover the local picture of $\trop(X)$ at $-v$. 
\end{enumerate}
\end{proof}

\begin{figure}
	\begin{tabular}{|c |c |c |c |c| p{3cm}|}
\hline
\# std. & \# anti-std. & \# non-std. & $\trop(X)$ & $\trop(X^*)$ & dual vertex type\\
\hline
&&&&&\\
3 & 0 & 0 &
\begin{tikzpicture}[scale=.4]
\draw (0,0) -- (-1,-1);
\draw (0,0) -- (0,1);
\draw (0,0) -- (1,0);
\end{tikzpicture}
&
\begin{tikzpicture}[scale=.4]
\draw [line width=.5mm] (0,0) -- (1,1);
\draw [line width=.5mm] (0,0) -- (0,-1);
\draw [line width=.5mm] (0,0) -- (-1,0);
\end{tikzpicture}
&
trivalent, all edges multiplicity two
\\
&&&&&\\
\hline
&&&&&\\
1 & 1 & 1 &
\begin{tikzpicture}[scale=.4]
\draw (0,0) -- (1,-1);
\draw (0,0) -- (0,1);
\draw (0,0) -- (-1,0);
\end{tikzpicture} &
\begin{tikzpicture}[scale=.4]
\draw [line width=.5mm] (0,0) -- (0,-1);
\draw  (0,0) -- (1,1);
\draw  (0,0) -- (-1,1);
\end{tikzpicture}
&
trivalent, one edge multiplicity two
\\
&&&&&\\
\hline
&&&&&\\
1 & 0 & 2 &
\begin{tikzpicture}[scale=.4]
\draw (0,0) -- (.5,-1);
\draw (0,0) -- (0,1);
\draw (0,0) -- (-1,1);
\end{tikzpicture} &
\begin{tikzpicture}[scale=.4]
\draw [line width=.5mm] (0,0) -- (0,-1);
\draw  (0,0) -- (1,-1);
\draw  (0,0) -- (-.5,1);
\draw  (0,0) -- (-1,0);
\draw  (0,0) -- (1,1);
\end{tikzpicture}
&five-valent
\\
&&&&&\\
\hline
&&&&&\\
0 & 1 & 2 &
\begin{tikzpicture}[scale=.4]
\draw (0,0) -- (-.5,1);
\draw (0,0) -- (0,-1);
\draw (0,0) -- (1,-1);
\end{tikzpicture} &
\begin{tikzpicture}[scale=.4]
\draw  (0,0) -- (-1,1);
\draw  (0,0) -- (.5,-1);
\draw (0,0) -- (-1,0);
\draw (0,0) -- (1,1);
\end{tikzpicture}
&four-valent
\\
&&&&&\\
\hline
&&&&&\\
0 & 0 & 3 &
\begin{tikzpicture}[scale=.4]
\draw (0,0) -- (1,1);
\draw (0,0) -- (0,-1);
\draw (0,0) -- (-1,0);
\end{tikzpicture} &
\begin{tikzpicture}[scale=.4]
\draw[fill] (0,0) circle [radius=0.15];
\end{tikzpicture}
&
marked
\\
&&&&&\\
\hline
\end{tabular}
\caption{Local contributions of vertices to $\trop(X^*)$}\label{figure:local}
\end{figure}

The second difficulty for viewing tropical duality as an involution on plane curves is that even when $\trop(X)$ is a smooth tropical curve, $\trop(X^*)$ is typically not smooth. However, our Theorem \ref{thm:dualcurve} describing the tropical dual curve is only applicable to smooth tropical curves. 
Thus, it would be very interesting to extend this theorem to include singular plane curves. Any such extension would have to take into account some information on the singularities of $X$, since for tropically singular $\trop(X)$, $\trop(X^*)$ is not determined by $\trop(X)$. One possible approach to account for singularities would be by defining a tropical version of Chern-Mather classes. In fact, we believe that the markings in our decorated tropical dual may be interpreted in this fashion.

\subsection{Bitangents}
Bitangents of algebraic plane curve are in bijection with ordinary nodes of its projective dual. Pl\"ucker's formula, together with the fact that a general plane curve $X$ of degree $d$ has $3d(d-2)$ flexes, implies that $X$ has \[\frac{(d+3)d(d-2)(d-3)}{2}\] bitangents. Similarly, we believe that tropical bitangents may be detected via projective duality. For example, let $X$ be a generic plane curve whose tropicalization is depicted on the left of Figure \ref{fig:bitangents}; the tropicalization of its dual $X^*$ is depicted on the right, where as usual, thick lines represent edges of multiplicity $2$. 
The point $p$ where two branches of $\trop(X^*)$ meet transversally corresponds to a tropical bitangent line $\Lambda$ of $\trop(X)$, depicted as a dashed line. By  \cite[Theorem 3.1]{LM17}, $\Lambda$ lifts to $4$ bitangents of $X$, which is exactly the intersection multiplicity of the two branches of $\trop(X^*)$ at $p$. We believe that this is more than just a coincidence.

\begin{conj}
Let $X$ be a hypersurface in $\PP^n$ such that $\trop(X)$ is smooth, and suppose that  $n$ faces of $\trop(X^*)$ meet transversally with multiplicity $k$ at a point $p$. Then there  are $k$ distinct $n$-tangent hyperplanes $H$ with  $\trop(H)=p$. 
\end{conj}

We prove the conjecture for plane curves that are generic in the sense of \cite[Assumption 3.3]{LM17}.
\begin{proof}[Proof sketch for generic plane curves]
Let $\Lambda$ be the tropical line dual to $p$. 
In order to use \cite[Theorem 3.1]{LM17}, we need to show that $\Lambda$ is tangent to $\trop(X)$ at two distinct points. 
Let $e',e''$ be the direction vectors of the two branches of $\trop(X^*)$ at $p$. Since $\trop(X)$ is assumed to be smooth, at least one the vectors, say $e'$,  must be in a standard direction. From the construction of the dual curve, $\Lambda$ is a shifted tangent of some point $q'$ along the direction $e'$. If $e''$ is in a standard direction as well, then $\Lambda$ is tangent to a point $q''$, shifted in direction $e''$ (similarly to Figure \ref{fig:bitangents}). Otherwise, $e''$ is not in a standard direction, and $\Lambda$ is tangent to $\trop(X)$ at $q''=-p$. In any case, the two tangency points $q'$ and $q''$ are distinct.

There are two cases to consider: either $e''$ is in a standard direction and $\Lambda$ is shifted from $q''$, or $e''$ is in a non-standard direction and $\Lambda$ is not shifted. In each of these cases, it is straightforward to check that the intersection multiplicity coincides with the lifting multiplicity from loc.~cit.

\end{proof}
\noindent We leave further investigation of this phenomenon for future work.

\begin{figure}
\centering
\begin{tikzpicture}[scale=.7]


\draw (-3,-1) -- (-1,1) -- (0,1) -- (1,0) -- (1,-1) -- (-1,-3);
\draw (-1,1) -- (-1,2);
\draw (0,1) -- (0,2);
\draw (1,0) -- (2,0);
\draw (1,-1) -- (2,-1);

\node at (0,-4) {$\trop(X)$};


\draw [blue,dashed] (1,1)--(-3,1);
\draw [blue,dashed] (1,1) -- (1,-3);
\draw [blue,dashed] (1,1)--(2,2);
{\tiny\node [blue, above left] at (1,1) {$\Lambda$};}

\begin{scope}[shift={(10,0)}, xscale=-1,yscale=-1]
\node at (0,4) {$\trop(X^*)$};
{\tiny \node [above left] at (1,1) {$p$}; }

\draw [line width=2] (-3,-1) -- (-1,1); 
\draw [line width=2] (-1,1) -- (3,1); 
\draw [line width=2] (-1,1) -- (-1,3); 
\draw [line width=2] (0,1) -- (0,3); 

\draw [line width=2] (1,0) -- (3,0); 
\draw [line width=2] (1,0) -- (1,-1); 
\draw [line width=2] (1,0) -- (1,3); 
\draw [line width=2] (1,-1) -- (3,-1); 
\draw [line width=2] (1,-1) -- (0,-2); 

\draw [] (0,1) -- (-2,-1); 
\draw [] (0,1) -- (1,0); 
\draw [] (1,0) -- (-1,-2); 

\draw[color=blue, fill=blue] (1,1) circle [radius=0.15];

\end{scope}

\end{tikzpicture}   

\caption{A bitangent of a tropical curve corresponding to a node of the dual.}
\label{fig:bitangents}

\end{figure}
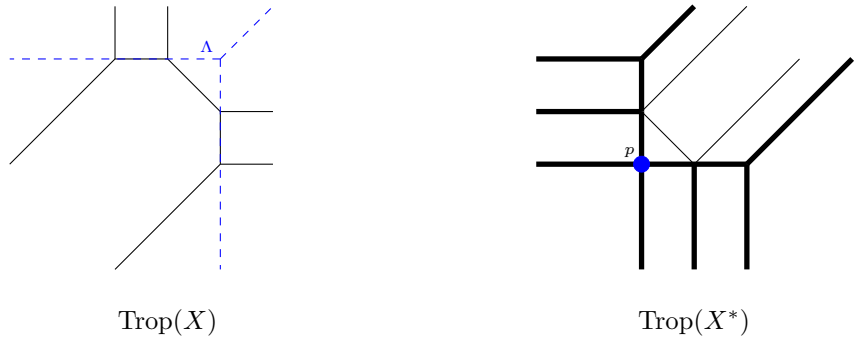

\subsection{Beyond surfaces}
It is natural to try to extend our results beyond the case of curves and surfaces. However, moving from curves to surfaces already led to a much more complicated result. We anticipate that for higher dimensional hypersurfaces, things become even more complicated.

Nonetheless, one could still try to apply the framework of \S \ref{sec:consistency} and \S \ref{sec:lifting} to describe the tropical dual of a hypersurface.
In general, looking at $k$-cells in $\trop(X)$ may require that we consider consistent sequences with $k+1$ terms in them, which will require understanding $\A_i$ and $\nu_i$ for $i\leq k$. While we have good control of $\A_1$ through \S\ref{sec:NOT}, understanding even $\A_2$ can be much more subtle. In the case of surfaces we managed to describe $\trop(X^*)$ without a complete understanding of $\A_2$; for $2$-faces in higher dimensional hypersurfaces, this will not be possible.

Instead of working with $\trop(X)$, one could instead take as input the tropicalization of the polynomial $f$. This would perhaps make understanding $\A_i$ and $\nu_i$ more straightforward. However, there is an even more fundamental difficulty:
Proposition \ref{prop:consistency} gives bounds on the valuations of not the $z_i$ (and thus the $y_i$), but on the linear forms $z_i^{(k)}$. One may use these bounds to deduce bounds on the $z_i$ as well (as we have done for curves and surfaces) but in higher dimensions, situations can occur in which the resulting inequalities on the $s_i$ give a subset of $\RR^n$ which has dimension $n$. This means that these inequalities alone do not always give sufficient criteria for the existence of a tropical tangent.

\subsection{Higher codimension}
Instead of increasing the dimension of $X$, it is natural to try to generalize our results by increasing the \emph{codimension} of $X$. On the one hand, the system of equations describing the conormal variety becomes more complicated. On the other hand, if we e.g.~restrict our attention to the case where $X$ is a curve, the combinatorial structure of $\trop(X)$ is still quite manageable. 

A natural starting point would be the case of space curves. This would be especially interesting from the point of view of tropical duality as an involution, since as we saw in Example \ref{ex:cone}, even for a smooth tropical surface $\trop(X)$ in $\RR^3$, $\trop(X^*)$ might be a curve.

\subsection{Multiplicities}
In \S\ref{sec:mult}, we have described a framework for calculating the multiplicities of $\trop(X^*)$, and applied it in the case of plane curves. It would be very interesting to extend this to a description of the multiplicities in the case with $X\subset \PP^3$  a surface. 

Armed with our description of $\trop(X^*)$ along with the multiplicities, one would be able to describe the Newton polytope of $X^*$. We believe that a (more complicated) version of the statement of Corollary \ref{cor:newton} should also hold for surfaces. Lemma \ref{lemma:constant} applies to hypersurfaces in arbitrary dimension, so given any hypersurface $X$ which is sufficiently generic with respect to its Newton polytope $\Delta$, there is some rule which will produce the Newton polytope of $X^*$; it is just a question of describing this rule explicitly.

\bibliographystyle{alpha}

\bibliography{paper}
\end{document}